\documentclass[12pt]{amsart}

\usepackage{enumerate}
\usepackage{a4wide}

\usepackage{etoolbox}
\patchcmd{\subsubsection}{\itshape}{\bfseries\itshape}{}{}

\usepackage{enumerate}

\newtheorem{theorem}{Theorem}[section]
\newtheorem{lemma}[theorem]{Lemma}
\newtheorem{proposition}[theorem]{Proposition}
\newtheorem{corollary}[theorem]{Corollary}
\newtheorem*{theoremA}{Theorem A}
\newtheorem*{theoremB}{Theorem B}
\newtheorem*{theoremC}{Theorem C}

\theoremstyle{remark}
\newtheorem{remark}[theorem]{Remark}

\newcommand{\C}{\ensuremath{\mathbb{C}}}
\newcommand{\R}{\ensuremath{\mathbb{R}}}
\renewcommand{\H}{\ensuremath{\mathbb{H}}}
\renewcommand{\O}{\ensuremath{\mathbb{O}}}

\newcommand{\g}[1]{\ensuremath{\mathfrak{#1}}}

\DeclareMathOperator{\tr}{tr}

\DeclareMathOperator{\Id}{Id}
\DeclareMathOperator{\Ad}{Ad}
\DeclareMathOperator{\Image}{Im}
\DeclareMathOperator{\Ker}{ker}
\DeclareMathOperator{\spann}{span}

\newcommand{\Sp}{\ensuremath{\mathsf{Sp}}}
\newcommand{\SO}{\ensuremath{\mathsf{SO}}}

\begin{document}
	\title[Homogeneous and inhomogeneous isoparametric hypersurfaces]{Homogeneous and inhomogeneous isoparametric hypersurfaces in rank one symmetric spaces}
	\author[J.~C. D\'\i{}az-Ramos]{Jos\'{e} Carlos D\'\i{}az-Ramos}
	\address{Department of Mathematics,
	University of Santiago de Compostela, Spain.}
	\email{josecarlos.diaz@usc.es}
	
	\author[M. Dom\'{\i}nguez-V\'{a}zquez]{Miguel Dom\'{\i}nguez-V\'{a}zquez}
	\address{Department of Mathematics,
	University of Santiago de Compostela, Spain.}
	\email{miguel.dominguez@usc.es}
	
	\author[A. Rodr\'{\i}guez-V\'{a}zquez]{Alberto Rodr\'{\i}guez-V\'{a}zquez}
	\address{Department of Mathematics,
		University of Santiago de Compostela, Spain.}
	\email{a.rodriguez@usc.es}
	\thanks{The authors have been supported by the projects MTM2016-75897-P (AEI/FEDER, Spain) and ED431C 2019/10, ED431F 2017/03 (Xunta de Galicia, Spain). The second and third authors  acknowledge  support  of  the  Ram\'on y Cajal program (Agencia Estatal de Investigaci\'on, Spain) and the FPU program (Ministry of Education, Spain), respectively}
	
	
	
	\begin{abstract}
		We conclude the classification of cohomogeneity one actions on symmetric spaces of rank one by classifying cohomogeneity one actions on quaternionic hyperbolic spaces up to orbit equivalence. As a by-product of our proof, we produce uncountably many examples of inhomogeneous isoparametric families of hypersurfaces with constant principal curvatures in quaternionic hyperbolic spaces.
	\end{abstract}
	
	
	\maketitle

\section{Introduction}

Riemannian geometry, in a very broad sense, can be understood as the study of those properties of a smooth manifold that are invariant under isometries.
Among Riemannian manifolds with large isometry groups, Riemannian symmetric spaces stand out as a class of their own, 
not only in Riemannian Geometry, but also in Lie group theory or Global Analysis.
In this class, Euclidean spaces and symmetric spaces of rank one are the most popular in Riemannian geometry.
It has been an interesting problem to study isometric actions on manifolds with large isometry groups, and several types of them have been investigated over the years. 
One of the most important families of isometric actions is that of cohomogeneity one, that is, proper isometric actions whose orbit space is one-dimensional, or in other words, whose
principal orbits are hypersurfaces. 
Cohomogeneity one actions have recently been of great interest for the construction of geometric structures, such as Einstein metrics, Ricci solitons, special holonomy, or minimal hypersurfaces, among~others.

However, it is also a natural and important problem to find all cohomogeneity one actions on a given Riemannian manifold, usually just up to orbit equivalence. This is a classical problem in submanifold geometry that traces back to the time of \'{E}.~Cartan, and which turns out to be equivalent to the classification of homogeneous hypersurfaces up to isometric congruence. By both historical and mathematical reasons, it has frequently been linked to the investigation of the so-called isoparametric hypersurfaces.

A hypersurface is called isoparametric if its nearby equidistant hypersurfaces have constant mean curvature. 
Thus, every homogeneous hypersurface is isoparametric.
In the 30s, Cartan himself studied the converse implication. 
Since all the examples known to him (which included all isoparametric hypersurfaces in Euclidean and real hyperbolic spaces, and all isoparametric hypersurfaces with up to three principal curvatures in spheres) were homogeneous, he posed the question: is it true that an isoparametric hypersurface is homogeneous? 
A surprising negative answer to this question arrived almost forty years later, with the construction of the first examples of inhomogeneous isoparametric hypersurfaces in spheres by Ozeki and Takeuchi~\cite{OzTa75}, soon generalized by Ferus, Karcher and M\"{u}nzner~\cite{FKM81}. 
These examples led to an added difficulty in the classification problem of isoparametric hypersurfaces in spheres, which has given rise to outstanding results over the last few years \cite{CCJ07,I08,Mi,Si16,Ch13,Ch18}. 
Other inhomogeneous isoparametric hypersurfaces have been found in symmetric spaces such as complex and quaternionic projective spaces \cite{Do16,DoGo18} and in certain symmetric spaces of non-compact type, such as complex hyperbolic spaces~\cite{DiDoSa17}, or more generally, the symmetric spaces with Dynkin diagram of ($BC_r$)-type \cite{damekricci,Do15}.
However, none of these examples, unlike the ones in spheres, have constant principal curvatures, with only one remarkable exception: one inhomogeneous family of isoparametric  hypersurfaces with constant principal curvatures in the Cayley hyperbolic plane~\cite{damekricci}.

The classification of cohomogeneity one actions up to orbit equivalence in Euclidean spaces follows from the classification of isoparametric hypersurfaces in $\R^n$ obtained by Segre~\cite{Se38}.
In symmetric spaces of compact type and rank one, the corresponding classification follows from several works.  In spheres it was obtained by Hsiang and Lawson~\cite{HsLa71}, in complex projective spaces by Takagi~\cite{Ta73}, and in quaternionic projective spaces and the Cayley plane by Iwata~\cite{Iw78,Iw81}.
There is also a classification of cohomogeneity one actions on irreducible symmetric spaces of compact type due to Kollross~\cite{Ko02}.

The problem is more difficult in the non-compact case.
The main reason is that, unlike in the compact setting, there are two main types (namely, reductive and parabolic) of maximal subgroups of the isometry group of a symmetric space of non-compact type, and parabolic subgroups contain many subgroups that act transitively on the space. Thus, the investigation of orbits of subgroups of a parabolic subgroup frequently leads to complicated linear algebra or combinatorial problems (in certain sense similar, for example, to the ones arising in the outstanding classification problem of totally geodesic submanifolds~\cite{klein}), for which very few ideas have been developed~(cf.~\cite{BT03, BT13, mathz}).
The first classification result of cohomogeneity one actions on a symmetric space of non-compact type was given by Cartan~\cite{Ca38} for real hyperbolic spaces, while he was studying isoparametric hypersurfaces in spaces of constant curvature. However, the classification in complex hyperbolic~spaces and the Cayley hyperbolic plane, due to Berndt and Tamaru~\cite{BT07}, only arrived seventy years later.
There are several structural results for symmetric spaces of non-compact type~\cite{BeDo15, BT13}, but a full classification is still not available, not even in quaternionic hyperbolic spaces. 

This is precisely the point where we start our study.
The main aim of this article is to classify cohomogeneity one actions on quaternionic hyperbolic spaces up to orbit equivalence.
Our method relies partially on the ideas developed in~\cite{BT07}, where it is proved that this classification can be reduced to a certain problem that we solve in this~paper.
\medskip

The first main result of this article can be stated in terms of quaternionic algebra.
We denote by $\H$ the real division algebra of the quaternions, endowed with its standard complex structures $i$, $j$ and $k$.
Let $\H^n$ be a right quaternionic vector space of dimension~$n$. The compact symplectic group $\Sp(n)$ is the group of quaternionic matrices (acting on the left on~$\H^n$) that preserve the standard quaternionic bilinear form $\sum_{i=1}^n \bar{v}_i w_i$, where $v,w\in\H^n$, and bar denotes conjugation.
This bilinear form naturally induces an inner product in $\H^n$ that makes it isometric with $\R^{4n}$. By $\g{J}$ we will denote the quaternionic structure of $\H^n$, that is, the subspace of real endomorphisms of $\H^n$ generated by the right multiplications by $i,j$ and $k$, which can therefore be seen as the Lie algebra of~$\Sp(1)$.

We also consider the Lie group $\Sp(1)\Sp(n)=\Sp(1)\times\Sp(n)/\mathbb{Z}_2$, which acts on $\H^n$ as $(q,A)\cdot v=Avq^{-1}$. This is an important group in Differential Geometry, as it arises in Berger's holonomy list, that is, the list of Lie groups which can be realized as the holonomy of irreducible, simply connected and non-locally symmetric Riemannian manifolds. Thus, a Riemannian manifold is called quaternionic K\"{a}hler if it has dimension $4n$, is not Ricci-flat, and its holonomy is isomorphic to a subgroup of $\Sp(1)\Sp(n)$, $n\geq 2$. The simplest examples of symmetric, quaternionic K\"ahler spaces are the quaternionic projective spaces, and their non-compact duals, the quaternionic hyperbolic spaces. In any case, understanding algebraic properties linked to holonomy groups is a first fundamental step towards the study of more geometric questions, such as those related to curvature (e.g.~the celebrated LeBrun-Salamon conjecture~\cite{LBSa94}) or submanifolds (e.g.~the theory of calibrations~\cite{BH:jams}). Similarly, the problem of submanifold geometry that we address in this paper relies on a linear algebraic problem that we describe below.

We say that a real subspace $V$ of $\H^n$ is \emph{protohomogeneous} if there exists a connected Lie subgroup of $\Sp(1)\Sp(n)$ that acts transitively on the unit sphere of $V$.
A protohomogeneous subspace of $\H^n$ has constant quaternionic K\"{a}hler angle, a concept that is central in our study and that we recall now.
Let $\pi_V$ denote the orthogonal projection onto a vector subspace $V$,  and define 
\[
P_J=\pi_V\circ J,\quad \text{where } J\in\g{J}.
\] 
We say that $V$ has \emph{constant quaternionic K\"{a}hler angle} $(\varphi_1,\varphi_2,\varphi_3)$, with $\varphi_1\leq \varphi_2\leq\varphi_3$, if for any $v\in V$ the symmetric bilinear form
\[
L_v\colon\g{J}\times\g{J}\to\R,\quad L_v(J,J')=\langle P_J v, P_{J'}v\rangle,
\]
has eigenvalues $\cos^2(\varphi_i)\langle v,v\rangle$, $i\in\{1,2,3\}$.
We point out here the fact that the bilinear forms $L_v$, $v\in V$, described above do not necessarily diagonalize simultaneously (although we can prove \textit{a priori} that they do so for protohomogeneous subspaces of dimension greater or equal than~5, see Corollary~\ref{prop:in_spn}, and by classification results for dimension different from~3).

The first main result of this article is to classify, up to congruence by elements in $\Sp(1)\Sp(n)$, protohomogeneous subspaces of $\H^n$.
We state here the moduli space of such subspaces of dimension $k$ in $\H^n$ by presenting their possible quaternionic K\"{a}hler angles.
In Theorem~A, and in what follows, $\sqcup$ denotes disjoint union.

\begin{theoremA}
The moduli space $\mathcal{M}_{k,n}$ of non-zero protohomogeneous subspaces of dimension $k$ in $\H^n$, up to congruence in $\Sp(1)\Sp(n)$, is described in the following table:
\medskip

{\renewcommand{\arraystretch}{1.5}
\begin{tabular}{lllll}
\hline
$\mathcal{M}_{k,n}$ & $k\leq n$  & $n<k\leq \frac{4n}{3}$
& $\frac{4n}{3}<k\leq 2n$  & $k>2n$
\\%
\hline
$k\equiv 0 \,\mathrm{(mod\; 4)}$   & $(\g{R}_4^+\setminus\g{R}_4^-)
\sqcup (\g{R}^-_4\times\mathbb{Z}_2)$ & $\g{S}$ & $\{(0,\varphi,\varphi)\}_{\varphi\in[0,\frac{\pi}{2}]}$ & $\{(0,0,0)\}$
\\%
$k\equiv 2 \,\mathrm{(mod\; 4)}$
& $\{(\varphi,\frac{\pi}{2},\frac{\pi}{2})\}_{ \varphi\in[0,\frac{\pi}{2}]}$ & $\{(0,\frac{\pi}{2},\frac{\pi}{2})\}$
& $\{(0,\frac{\pi}{2},\frac{\pi}{2})\}$
& $\emptyset$
\\%
$k\neq 3$ \textup{odd}
& $\{(\frac{\pi}{2},\frac{\pi}{2},\frac{\pi}{2})\}$
& $\emptyset$ & $\emptyset$ & $\emptyset$
\\%
$k=3$ & $(\g{R}_3^+\setminus\g{R}_3^-) \sqcup (\g{R}^-_3\times\mathbb{Z}_2)$
& $\emptyset$
& $\{(\varphi,\varphi,\frac{\pi}{2})\}_{\varphi\in\{0,\frac{\pi}{3}\}}$ & $\{(0,0,\frac{\pi}{2})\}$
\\%
\hline
\end{tabular}}
\medskip

where $\Lambda=\{(\varphi_1,\varphi_2, \varphi_3)\in [0,\pi/2]^3: \varphi_1\leq \varphi_2\leq \varphi_3 \}$, and
\begin{align*}
\mathfrak{R}_3^+
&{}=\{ (\varphi,\varphi,\pi/2)\in \Lambda :\varphi\in[0,\pi/2]  \},\\
\mathfrak{R}_3^-
&{}=\{ (\varphi,\varphi,\pi/2)\in \Lambda :\varphi\in[\pi/3,\pi/2)   \},\\
\mathfrak{R}^+_4
&{}=\{  (\varphi_1,\varphi_2,\varphi_3)\in \Lambda :\cos(\varphi_1)+\cos(\varphi_2)-\cos(\varphi_3)\leq 1 \},\\
\mathfrak{R}^-_4
&{}=\{  (\varphi_1,\varphi_2,\varphi_3)\in \Lambda :\cos(\varphi_1)+\cos(\varphi_2)+\cos(\varphi_3)\leq 1,\,\varphi_3\neq \pi/2 \},\\
\mathfrak{S}
&{}=\{(\varphi_1,\varphi_2,\varphi_3)\in \Lambda : \cos(\varphi_1)+\cos(\varphi_2)+\varepsilon\cos(\varphi_3)=1,
\text{ for $\varepsilon=1$ or $\varepsilon=-1$}\}.
\end{align*}
\end{theoremA}

This classification includes typical examples such as totally real subspaces (precisely those with quaternionic K\"{a}hler angle $(\pi/2,\pi/2,\pi/2)$), totally complex subspaces (with quaternionic K\"{a}hler angle $(0,\pi/2,\pi/2)$), quaternionic subspaces (with quaternionic K\"{a}hler angle $(0,0,0)$), subspaces of constant K\"{a}hler angle $\varphi\in(0,\pi/2)$ inside a totally complex vector subspace (with quaternionic K\"{a}hler angle $(\varphi,\pi/2,\pi/2)$), complexifications of subspaces of constant K\"{a}hler angle $\varphi\in(0,\pi/2)$ in a totally complex subspace (with quaternionic K\"{a}hler angle $(0,\varphi,\varphi)$), and $(\mathop{\rm Im}\H)v$, $v\in\H^n$, $v\neq 0$ (with quaternionic K\"{a}hler angle $(0,0,\pi/2)$).
However, there are some other non-classical examples.  Some of them were introduced in~\cite{damekricci}, but there are some others, which are basically presented and classified in Section~\ref{sec:low-dim}.
A basis of these subspaces can be calculated explicitly, but for $\g{R}_3^\pm$ and $\g{R}_4^\pm$ its expression is rather long.  See Proposition~\ref{prop:sk=3} for $\g{R}_3^\pm$ and Propositions~\ref{prop:class_dim4} and~\ref{prop:protosum} for $\g{R}_4^\pm$ to get further details.
Furthermore, there are non-congruent subspaces of $\H^n$ with the same K\"{a}hler angles. These correspond precisely to the intersections $\g{R}_3^+\cap\g{R}_3^-=\g{R}_3^-$ and $\g{R}_4^+\cap\g{R}_4^-=\g{R}_4^-$.

We point out here three main tools that have been essential to obtain this classification.
First we use the classical generalization of the hairy ball theorem regarding the possible rank of continuous distributions on spheres~\cite{Steenrod} in order to reduce the classification problem of real subspaces of $\H^n$ with constant quaternionic K\"{a}hler angle to subspaces of dimensions 3 and multiples of~$4$ (Section~\ref{sec:hairy-ball}).
Secondly, we provide a Lie theoretic argument relying on results by Borel~\cite{Bo50} and Montgomery and Samelson~\cite{MS43} on groups acting effectively and transitively on spheres, to prove that,
for subspaces of dimension greater or equal than~5,
the maps $L_v$ that are used to define quaternionic K\"{a}hler angle diagonalize simultaneously (Corollary~\ref{prop:in_spn}).
In third place, using the previous results, we can show that a protohomogeneous subspace of dimension $4l$ is the sum of protohomogeneous subspaces of dimension $4$ with the same quaternionic K\"{a}hler angle (Section~\ref{subsec:factorization}).
All this reduces the classification of protohomogeneous subspaces to dimensions~$3$ and~$4$. At this stage, we actually obtain the more general classification of real subspaces of dimensions $3$ and $4$ with constant quaternionic K\"ahler angle.
This is a (hard) problem of linear algebra that is solved in Section~\ref{sec:low-dim}. 
\medskip

The first consequence of Theorem~A is the classification of cohomogeneity one actions on quaternionic hyperbolic spaces $\H H^{n+1}$ up to orbit equivalence.
In fact, Berndt and Tamaru explained in~\cite{BT07} how to obtain this classification.
Consider the symmetric pair $(G,K)=(\Sp(1,n+1),\Sp(1)\times\Sp(n+1))$ representing the symmetric space $\H H^{n+1}$. 
We denote by $\g{g}=\g{k}\oplus\g{p}$ the corresponding Cartan decomposition, and let $\g{a}$ be a maximal abelian subspace of $\g{p}$, which is one-dimensional because $\H H^{n+1}$ is of rank one.
Let $\g{g}=\g{g}_{-2\alpha}\oplus\g{g}_\alpha
\oplus\g{g}_0\oplus\g{g}_\alpha\oplus\g{g}_{2\alpha}$ be the restricted root space decomposition of $\g{g}$ with respect~to~$\g{a}$.~Then, $\g{g}_\alpha$ is isomorphic to a quaternionic vector space $\H^n$ endowed with the standard quaternionic bilinear form, and $K_0\cong \Sp(1)\times\Sp(n)$, the connected Lie subgroup of $G$ whose Lie algebra is $\g{k}_0=\g{g}_0\cap \g{k}=N_{\g{k}}(\g{a})$, normalizes $\g{g}_\alpha$ and acts on $\g{g}_\alpha$ in the canonical way.
The classification of cohomogeneity one actions on $\H H^{n+1}$ can be obtained if we determine the protohomogeneous subspaces $V$ of $\g{g}_\alpha\cong\H^n$.
If $V$ is such a protohomogeneous subspace, we define the Lie subalgebra $\g{s}_V=\g{a}\oplus (\g{g}_\alpha\ominus V)\oplus\g{g}_{2\alpha}$ of $\g{g}$, and denote by $S_V$ the connected Lie subgroup of $G$ with Lie algebra $\g{s}_V$
(throughout this article $\ominus$ denotes the orthogonal complement of a vector subspace).
Then $N_{K_0}^0(S_V)S_V=N_{K_0}^0(V)S_V$ acts on $\H H^{n+1}$ with cohomogeneity one,
where $N_{K_0}^0(\cdot)$ denotes the connected component of the identity of the normalizer in $K_0$.
Knowing all such subspaces $V$ up to congruence by an element of $\Sp(1)\Sp(n)$ determines all cohomogeneity one actions on $\H H^{n+1}$ up to orbit equivalence.

Roughly twenty years after Berndt and Br\"uck~\cite{BB01} announced the first examples of cohomogeneity one actions using this procedure, we obtain the full classification of cohomogeneity one actions on quaternionic hyperbolic spaces up to orbit equivalence as a consequence of Theorem~A.
Together with the results by Berndt and Tamaru~\cite{BT07}, this finishes the classification of cohomogeneity one actions on non-compact symmetric spaces of rank one:

\begin{theoremB}
The moduli space of cohomogeneity one actions on $\H H^{n+1}$ up to orbit equivalence is given by the disjoint union
\[
\{N, K, \mathsf{SU}(1,n+1)\} \sqcup\bigsqcup_{k=1}^{4n} \mathcal{M}_{k,n}.
\]
The actions referenced here are:
\begin{enumerate}[{\rm (1)}]
\item $N$: the action that produces a horosphere foliation.
\item $K$: the action that produces a family of geodesic spheres centered at a point.
\item $\mathsf{SU}(1,n+1)$: the action that produces a family of tubes around a totally geodesic $\C H^{n+1}$.
\item $\mathcal{M}_{k,n}$: the cohomogeneity one actions of the connected Lie subgroups of $\Sp(1,n+1)$ with Lie algebras $N_{\g{k}_0}(V)\oplus\g{a}\oplus(\g{g}_\alpha\ominus V)\oplus\g{g}_{2\alpha}$, where $V$ is a protohomogeneous subspace of dimension~$k$ of $\g{g}_\alpha\cong \H^n$.\label{th:B:Mkn}
\end{enumerate}
\end{theoremB}

We note that, in this classification, the action of $\Sp(1,\ell) \times \Sp(n+1-\ell)\subset \Sp(1,n+1)$ which gives tubes around a totally geodesic lower dimensional quaternionic hyperbolic space $\H H^\ell$, $\ell\in\{1,\dots,n\}$, in $\H H^{n+1}$ are included in item (\ref{th:B:Mkn}), where in this case $V$ is a quaternionic subspace of $\g{g}_\alpha\cong\H^n$ (hence, of quaternionic K\"{a}hler angle $(0,0,0)$) of real dimension $k=4(n-\ell+1)$.
Moreover, if we take $V$ a line in $\g{g}_\alpha$ (i.e.\ $k=1$), then $N_{K_0}^0(V)$ is trivial and we recover the action that gives rise to the so-called solvable foliation~\cite{BT03}.
\medskip

In our study of protohomogeneous subspaces of $\H^n$ we have also encountered non-congruent pairs of subspaces with the same constant quaternionic K\"{a}hler angles.
Moreover, we prove in Section~\ref{inhomo} that an $\H$-orthogonal direct sum of subspaces of dimension~4 with the same constant quaternionic K\"{a}hler angle is protohomogeneous if and only if any two factors are congruent under an element of $\Sp(n)$.
However, even if that direct sum is not protohomogeneous, it has constant quaternionic K\"{a}hler angle in some cases.
Thus, if we take $V$ a non-protohomogeneous subspace with constant quaternionic K\"{a}hler angle as above, and denote by $S_V$ the subgroup of $G$ whose Lie algebra is $\g{s}_V=\g{a}\oplus(\g{g}_\alpha\ominus V)\oplus\g{g}_{2\alpha}$, then: (1) since $V$ has constant quaternionic K\"{a}hler angle, tubes around $S_V\cdot o$ are isoparametric and have constant principal curvatures by~\cite[Theorem 4.5]{damekricci}, and (2) these tubes are not homogeneous by~\cite[Theorem 4.1]{BT07}. Hence, we have the following remarkable consequence:

\begin{theoremC}
There exist uncountably many inhomogeneous isoparametric families of hypersurfaces with constant principal curvatures in $\H H^{n+1}$ with $n\geq 7$, up to congruence.
\end{theoremC}

We recall that the only examples of inhomogeneous isoparametric families of hypersurfaces with constant principal curvatures known so far in any irreducible Riemannian symmetric space are the celebrated examples in spheres by Ferus, Karcher and M\"unzner~\cite{FKM81} and a single example found in the Cayley hyperbolic plane~\cite{damekricci}. Thus, this is the first time an uncountable collection of such examples is produced in some symmetric space.

This article is organized as follows. We recall some basic facts about symmetric spaces in \S\ref{subsec:symmetric_spaces}, and of cohomogeneity one actions in \S\ref{sec:cohomo}.
The fundamental concept of quaternionic K\"{a}hler angle is recalled in Subsection~\ref{subsec:kahlerangle} together with some important notation that will be used throughout this article.
In Section~\ref{sec:hairy-ball} we use a generalization of the hairy ball theorem to rule out several possibilities for quaternionic K\"{a}hler angles.
Then, in Subsection~\ref{subsec:canonical} we prove a simultaneous diagonalization result for subspaces of constant quaternionic K\"{a}hler angle.  This is used in \S\ref{subsec:factorization} to prove a factorization theorem for protohomogeneous subspaces of dimension multiple of $4$.
Altogether, this reduces our study to dimensions~$3$ (\S\ref{subsec:dim3}) and~$4$ (\S\ref{subsec:dim4}).
The existence of inhomogeneous isoparametric hypersurfaces with constant principal curvatures in quaternionic hyperbolic spaces (Theorem~C) is established in Section~\ref{inhomo}.
We finally prove Theorems~A and~B in Section~\ref{sec:proof}.

\section{Preliminaries}\label{secPreliminaries}
	
	We start this section by recalling the main known results concerning cohomogeneity one actions on symmetric spaces of non-compact type and rank one. Cohomogeneity one actions with a non-totally geodesic singular orbit are built using the concept of {quaternionic K\"ahler angle}, which we recall in this section. Also, we will present some properties and sumarize all the examples of subspaces with constant quaternionic K\"ahler angle known up to the present. The main references for these notions and results are~\cite{BB01}, \cite{BT07}, and~\cite{damekricci}.
	
\subsection{Symmetric spaces of non-compact type and rank one}\label{subsec:symmetric_spaces}\hfill
	
	Hurwitz's theorem asserts that any normed real division algebra $\mathbb{F}$ is isomorphic to $\R$, $\C$, $\H$ or $\O$. The hyperbolic spaces over these algebras constitute the symmetric spaces of non-compact type and rank one. In other words, if $M$ is a symmetric space of non-compact type and rank one, then $M$ is either a real hyperbolic space $\R H^{n+1}$, $n\geq 1$, a complex hyperbolic space $\C H^{n+1}$, $n\geq 1$, a quaternionic hyperbolic space $\H H^{n+1}$, $n\geq 1$, or the Cayley hyperbolic plane $\O H^2$. As a symmetric space, any of these manifolds $M$ can be identified with a quotient $G/K$ of Lie groups, where $G$ is the connected component of the identity of the isometry group of $M$, up to a finite covering, and $K$ is the isotropy subgroup of $G$ corresponding to a certain point $o\in M$ that we fix from now on. Then one can take $G=\mathsf{SO}^0(1,n+1)$, $\mathsf{SU}(1,n+1)$, $\Sp(1,n+1)$, $\mathsf{F}_{4}^{-20}$ and $K=\mathsf{SO}(n+1)$, $\mathsf{S}(\mathsf{U}(1)\times\mathsf{U}(n+1))$, $\Sp(1)\times\Sp(n+1)$, $\mathsf{Spin}(9)$, depending on whether $\mathbb{F}=\R$, $\C$, $\H$, $\O$, respectively.
	
	We denote by  $\g{g}$ and $\g{k}$ the Lie algebras of $G$ and $K$, respectively, by $\mathcal{B}$ the Killing form of $\g{g}$, and by $\theta$ the Cartan involution of $\g{g}$ with respect to $\g{k}$.  Let $\g{g}=\g{k} \oplus\g{p}$ be the Cartan decomposition of $\g{g}$ induced by $\theta$. We have that  $\langle X, Y\rangle=-\mathcal{B}(X,\theta Y)$ is an inner product that restricted to $\g{p}$ induces a Riemannian metric on $G/K$ that makes $G/K$ isometric to $M$, up to homothety.
	
	Let $\g{a}$ be a maximal abelian subspace of~$\g{p}$, which is one-dimensional as  $M$ has rank one, and let
$
\g{g}=\g{g}_{-2\alpha}\oplus\g{g}_{\alpha} \oplus\g{g}_{0} \oplus \g{g}_{\alpha}\oplus\g{g}_{2\alpha}
$
be the corresponding restricted root space decomposition of $\mathfrak{g}$. Here, the root space $\g{g}_0$ splits as $\g{g}_0=\g{k}_0\oplus\g{a}$, where $\g{k}_0$ is the Lie algebra of $K_0=N_K(\g{a})$, the normalizer of $\g{a}$ in $K$, which also normalizes $\g{g}_\alpha$
and centralizes~$\g{g}_{2\alpha}$. Moreover,
$\g{g}=\g{k} \oplus \g{a} \oplus\g{n}$, where $\g{n}= \g{g}_{\alpha}\oplus\g{g}_{2\alpha}$, is an Iwasawa decomposition of $\g{g}$. When $\mathbb{F}=\R$, we have $\g{g}_{-2\alpha}=\g{g}_{2\alpha}=0$ and $\g{n}$ is abelian. Otherwise, $\g{n}$ is only two-step nilpotent. In fact, $\g{n}$ is isomorphic to the $(2n+1)$-dimensional Heisenberg algebra when $\mathbb{F}=\C$ and to a certain generalized Heisenberg algebra if $\mathbb{F}\in\{\H,\O\}$ (see \cite{BTV95}). Moreover, $\g{g}_{2\alpha}$, the center of $\g{n}$, is equal to the derived algebra of $\g{n}$, and has dimension $1$, $3$ or $7$ for $\mathbb{F}=\C$, $\H$ or $\O$, respectively.  In addition to this, we can identify $\g{g}_{\alpha}$ with $\R^{n}$, $\C^{n}$, $\H^{n}$, $\O$ for $\mathbb{F}=\R$, $\C$, $\H$, $\O$, respectively.  Indeed, $\g{g}_\alpha$ is a Clifford module over $\textsf{Cl}(m)$, where $m=\dim\g{g}_{2\alpha}$, which is the sum of equivalent Clifford modules if $m=3$, and is irreducible if $m=7$.
	
	The subalgebra $\g{a} \oplus \g{n}$ of $\g{g}$ is solvable and $\g{n}$ is its derived subalgebra.
	We denote by $A$ and by $N$ the connected closed Lie subgroups of $G$ with Lie algebras $\g{a}$ and $\g{n}$, respectively. Then, $G=KAN$ is an Iwasawa decomposition of $G$ and $AN$ is diffeomorphic to $M$. Furthermore, if we pull back the metric on  $M$  to $AN$ we get a left-invariant Riemannian metric on $AN$. Thus, $M$ is isometric to the solvable Lie group $AN$ endowed with a left-invariant metric. As such, it is an example of a Damek-Ricci space (see~\cite{BTV95}).

\subsection{Cohomogeneity one actions on hyperbolic spaces}\label{sec:cohomo}\hfill

	We can distinguish three different classes of cohomogeneity one actions on symmetric spaces of non-compact type and rank one, up to orbit equivalence. It was shown in \cite{BB01} that any such action has at most one singular orbit.
	
	\subsubsection*{Actions with no singular orbit}\hfill

	Berndt and Tamaru~\cite{BT03} classified actions without singular orbits. They proved that there are exactly two such actions up to orbit equivalence.
	\begin{enumerate}[{\rm (i)}]
		\item The action of $N$ on $\mathbb{F}H^{n+1}$ has cohomogeneity one. The orbits of this action are mutually congruent horospheres which form a regular Riemannian foliation on $\mathbb{F}H^{n+1}$, called the horosphere foliation.
		\item Let $S$ be the connected Lie subgroup of $AN$ with Lie algebra $\g{s}=\g{a} \oplus \g{w} \oplus \g{z}$, where $\g{w}$ is a vector subspace of  $\g{g}_\alpha$ of codimension one. Different choices of $\g{w}$ lead to conjugate actions. The action of $S$ on $\mathbb{F} H^{n+1}$ has cohomogeneity one and its orbits form a regular Riemannian foliation on $\mathbb{F}H^{n+1}$, called the solvable foliation.
	\end{enumerate}

	\subsubsection*{Actions with a totally geodesic singular orbit}\hfill

	Berndt and Br\"uck~\cite{BB01} classified cohomogeneity one actions on $\mathbb{F}H^{n+1}$ with a totally geodesic singular orbit $F$. The remaining orbits, which are principal, are tubes around the totally geodesic submanifold $F$, where:
	\begin{enumerate}[{\rm (i)}]
		\item $\mathbb{F}=\mathbb{R}$: $F\in \{\text{point},\mathbb{R}H^1,\ldots, \mathbb{R}H^{n-1} \}$;
		\item $\mathbb{F}=\mathbb{C}$: $F\in \{\text{point},\mathbb{C}H^1,\ldots, \mathbb{C}H^{n},\mathbb{R}H^{n+1} \}$;
		\item $\mathbb{F}=\mathbb{H}$: $F\in \{\text{point},\mathbb{H}H^1,\ldots, \mathbb{H}H^{n},\mathbb{C}H^{n+1} \}$;
		\item $\mathbb{F}=\mathbb{O}$: $F\in \{\text{point},\mathbb{O}H^1, \mathbb{H}H^{2} \}$.
	\end{enumerate}

	\subsubsection*{Actions with a non-totally geodesic singular orbit}\hfill

	Berndt and Tamaru~\cite{BT07} gave a construction method of all cohomogeneity one actions with a non-totally geodesic singular orbit in hyperbolic spaces. Such actions only appear if $\mathbb{F}\neq\R$. 
	We recall that $K_0$ acts on $\g{g}_\alpha$ by the adjoint representation, and hence, if $V$ is a real subspace of $\g{g}_\alpha$,  $N^0_{K_0}(V)$ will denote the connected component of the identity of the normalizer of $V$ in $K_0$.
	
\begin{theorem}\label{th:bt:classification}
Let $\g{g}=\g{k}\oplus\g{a}\oplus\g{n}$ be an Iwasawa decomposition of the Lie algebra of the isometry group of the hyperbolic space $M=\mathbb{F}H^{n+1}$, $\mathbb{F}\in\{\mathbb{C},\mathbb{H},\mathbb{O}\}$.
\begin{enumerate}[{\rm (i)}]
\item Let $V$ be a non-zero vector subspace of $\g{g}_\alpha$ such that $N^0_{K_0}(V)$ acts transitively on the unit sphere of $V$. Denote by $\g{g}_\alpha\ominus V$ the orthogonal complement of $V$ in $\g{g}_\alpha$. Then the connected subgroup of $G$ with Lie algebra $N_{\g{k}_0}(V)\oplus\g{a}\oplus(\g{g}_\alpha\ominus V)\oplus\g{g}_{2\alpha}$ acts on $M$ with cohomogeneity one, and the orbit through $o$ is singular, provided that $\dim V\geq 2$. Furthermore, every cohomogeneity one action on $M$ with a non-totally geodesic singular orbit can be obtained in this way up to orbit equivalence.\label{th:bt:i}
\item Let $V$ and $V'$ be vector subspaces of $\g{g}_\alpha$ as in \emph{(\ref{th:bt:i})}, and assume that the corresponding cohomogeneity one actions have non-totally geodesic singular orbits. Then, these actions are orbit equivalent if and only if there exists $k\in K_0$ such that $\Ad(k)V = V'$.
\end{enumerate}
\end{theorem}

\subsection{Quaternionic K\"ahler angle}\label{subsec:kahlerangle}\hfill

	The metric and the quaternionic K\"ahler structure on $\mathbb{H}H^{n+1}$ induce a positive definite inner product $\langle\cdot,\cdot\rangle$ on $\g{g}_\alpha$ and a quaternionic structure $\mathfrak{J}$ on $\g{g}_\alpha$, respectively, such that $\g{g}_\alpha$ is isomorphic to $\mathbb{H}^{n}$ as a (right) quaternionic Euclidean space. Here, by a quaternionic structure $\g{J}$ we understand a $3$-dimensional vector subspace of  $\mathrm{End}_\R(\H^n)$, the space  of real endomorphisms of $\H^n\cong\R^{4n}$, admitting a basis $\{J_1,J_2,J_3\}$ of orthogonal transformations of $\H^n\cong\R^{4n}$ such that $J_i^2=-\Id$ and $J_iJ_{i+1}=J_{i+2}=-J_{i+1}J_i$, for each $i\in\{1,2,3\}$ (indices modulo~3).
Such a basis is called a \emph{canonical basis} of the quaternionic structure~$\g{J}$.
Sometimes it is helpful to regard $\g{J}$ as endowed with a positive definite inner product that makes it isometric to the Euclidean $3$-space $\R^3$, and such that the elements of $\g{J}$ that are orthogonal complex structures of $\H^n$ constitute the unit sphere $\mathbb{S}^2\subset\g{J}$ with respect to such inner product.
Throughout this article, if $v$ is a vector in $\H^n$ and $V$ is a real subspace of $\H^n$ (i.e.\ a vector subspace of the real vector space $\R^{4n}$ with the underlying real vector space structure of $\H^n$), we will denote by $\H v=\R v\oplus \g{J}v$ and by $\H V=V+\g{J}V$ the quaternionic spans of $v\in \H^n$ and of $V\subset\H^n$, respectively; sometimes we will also write $(\mathrm{Im}\, \H)v$ to refer~to~$\g{J}v$.
	
	Theorem~\ref{th:bt:classification} shows the crucial role played by real subspaces $V$ of $\g{g}_\alpha\cong\H^{n}$  and their behavior with respect to $K_0$ in the classification problem of cohomogeneity one actions on $\H H^{n+1}$. Note that the effectivization of  $K_0$ on $\g{g}_\alpha\cong\H^n$ is the Lie group $\Sp(1)\Sp(n)=(\Sp(1)\times \Sp(n))/\{\pm(1,\Id)\}$, which acts in the standard way: $(q,A)\cdot v=Avq^{-1}$, where $q\in \Sp(1)$ and $A\in\Sp(n)$. Thus, in this subsection we gather some important terminology and useful facts to study real subspaces of a quaternionic Euclidean space, up to congruence by elements of $\Sp(1)\Sp(n)$.

	Firstly, motivated by Theorem~\ref{th:bt:classification}, we will say that a real subspace $V\subset \H^n$ is \textit{protohomogeneous} if there is a connected subgroup of $\Sp(1)\Sp(n)$ that acts transitively on the unit sphere of~$V$. Equivalently, $V$ is protohomogeneous if the connected Lie group $N^0_{\Sp(1)\Sp(n)}(V)$ acts transitively on the unit sphere of $V$. Note that protohomogeneous subspaces $V$ of $\g{g}_\alpha\cong\H^n$ are precisely those inducing cohomogeneity one actions on $\H H^{n+1}$ via the construction in Theorem~\ref{th:bt:classification}(i). We will also say that two real subspaces $V$ and $W$ of $\H^n$ are \emph{equivalent} if there exists an element $T\in \Sp(1)\Sp(n)$ such that $T V= W$. Observe that, by Theorem~\ref{th:bt:classification}(ii), $V$ and $W$ are equivalent protohomogeneous subspaces of $\g{g}_\alpha\cong\H^n$ if and only if they induce orbit equivalent cohomogeneity one actions on $\H H^{n+1}$.
	
	Let us now recall a useful description of the action of $\Sp(1)\Sp(n)$ on $\H^{n}$. Let $\{X_1,\ldots,X_{n}\}$ and $\{Y_1,\ldots,Y_{n}\}$ be two $\H$-orthonormal bases of $\H^{n}$, and let $\{J_1, J_2, J_3\}$ and $\{J_1', J_2', J_3'\}$ be two canonical bases of the quaternionic structure of $\H^{n}$. 
Then, there exists a unique $T\in \Sp(1)\Sp(n)$ such that
	$T(X_i)=Y_i$ and $T J_j=J_j' T$ for all $i\in\{1,\ldots, n\}$ and all $j\in\{1,2,3\}$. Conversely, any $\R$-linear endomorphism of $\H^{n}$ which maps $\H$-orthonormal bases of $\H^{n}$ to $\H$-orthonormal bases of $\H^{n}$ and intertwines canonical bases of the quaternionic structure of $\H^{n}$ in the above described fashion lies in $\Sp(1)\Sp(n)$.
	
	Let $V$ be a real vector subspace of the quaternionic Euclidean space $\H^n$. The \textit{K\"ahler angle} of a non-zero vector  $v\in V$ with respect to a non-zero $J\in \mathfrak{J}$  and $V$ is defined to be the angle between $J v$ and $V$. Equivalently, it is the value $\varphi \in [0,\pi/2]$ such that $\langle P_J v, P_J v \rangle =\cos^2(\varphi)\langle v,  v \rangle$, where $P_J:=\pi_V J$ and we denote by $\pi_V$ the orthogonal projection onto $V$.

	The following lemma was essentially proved by Berndt and Br\"uck \cite[Lemma~3]{BB01}. We state it in a somewhat different form following~\cite[Theorem~3.1]{damekricci}, where it was proved in the more general context of subspaces of Clifford modules.
	
\begin{lemma} \label{lemma:qKa}
Let $V$ be a real subspace of $\H^n$ and let $v\in V$ be a non-zero vector. Then there exists a canonical basis $\{J_{1},J_2, J_3\}$ of $\mathfrak{J}$ and a uniquely defined triple $(\varphi_1,\varphi_2,\varphi_3)$,  such that:
\begin{enumerate}[{\rm (i)}]
\item $\varphi_i$ is the K\"ahler angle of $v$ with respect to $J_{i}$ for each $i\in\{1,2,3\}$,
\item $\langle P_{i}v,P_{j}v\rangle=0$ for every $i\neq j$, where $P_i =\pi_V J_i$.
\item $\varphi_1\leq\varphi_2\leq\varphi_3$.
\item $\varphi_1$ is minimal and $\varphi_3$ is maximal among the K\"ahler angles of $v$ with respect to all non-zero elements of $\mathfrak{J}$.
\end{enumerate}
Indeed, $\{J_{1},J_2, J_3\}$ is a basis of $\mathfrak{J}$ with respect to which the symmetric bilinear form
\[
L_v\colon\mathfrak{J}\times\mathfrak{J}\to\R,\qquad L_v(J,J'):=\langle P_J v, P_{J'}v\rangle,
\]
has a diagonal matrix expression with eigenvalues $\cos^2(\varphi_i)\langle v,v\rangle$, $i\in\{1,2,3\}$.
\end{lemma}

The previous lemma allows us to introduce the following definition~\cite{BB01}. If $V$ is a real subspace of $\H^n$, the \emph{quaternionic K\"ahler angle} of a non-zero vector $v\in V$ with respect to $V$ is the triple $(\varphi_1,\varphi_2,\varphi_3)$ given in Lemma~\ref{lemma:qKa}. Sometimes we will also say that $v\in V$ has quaternionic K\"ahler angle $(\varphi_1,\varphi_2,\varphi_3)$ with respect to $V$ and to the canonical basis $\{J_1,J_2,J_3\}$ of $\g{J}$, in order to specify that the basis $\{J_1,J_2,J_3\}$ is under the conditions of Lemma~\ref{lemma:qKa}. A linear subspace $V$ of $\H^n$ is said to have \emph{constant quaternionic K\"ahler angle} $\Phi(V)=(\varphi_1,\varphi_2,\varphi_3)$ if the triple $(\varphi_1,\varphi_2,\varphi_3)$ is independent of the non-zero (or by linearity, unit) vector $v\in V$. In this work, whenever we use the notation $\Phi(V)$ we will implicitly assume that $V$ has constant quaternionic K\"ahler angle.

\begin{remark}
Note that the $J_i \in \g{J}$ defined in Lemma~\ref{lemma:qKa} may depend on $v\in V$.
This is true, even in the case that $V$ has constant quaternionic K\"{a}hler angle.
For example $V=\mathop{\rm Im}\H\subset \H$ has constant quaternionic K\"{a}hler angle $\Phi(V)=(0,0,\pi/2)$, but the basis $\{J_1,J_2,J_3\}$ of Lemma~\ref{lemma:qKa} cannot be chosen independently of $v\in V$.
However, we will prove that, under certain hypotheses (see Corollary~\ref{prop:in_spn} or Proposition~\ref{prop:class_dim4}), the $J_i$ can be chosen independently of $v\in V$.  This is one of the crucial results in this article.
\end{remark}

The following result is known (see~\cite[p.~229]{BB01}), but we find it instructive to include~a~proof.

\begin{lemma}\label{lemma:protohomogeneous}
Let $V\subset \H^n$ be a protohomogeneous subspace. Then, $V$ has constant quaternionic K\"ahler angle.
\end{lemma}

\begin{proof}
Let $v\in V$ be a unit vector of quaternionic K\"ahler angle $(\varphi_1,\varphi_2,\varphi_3)$ with respect to $V$ and a canonical basis $\{J_1,J_2,J_3\}$ of $\g{J}$.
Thus, $\langle P_i v, P_j v\rangle=\cos^2(\varphi_i) \delta_{ij}$, for $i,j\in \{1,2,3\}$, where $\delta_{ij}$ stands for Kronecker delta. Let $w\in V$ be a unit vector. Since $V$ is protohomogeneous, there exists $T\in \Sp(1)\Sp(n)$ which leaves $V$ invariant and satisfies $Tv=w$. By the description of the action of $\Sp(1)\Sp(n)$ on $\H^n$, there exists a canonical basis $\{J_1',J_2',J_3'\}$ of $\g{J}$ such that $TJ_i=J_i'T$, for $i\in\{1,2,3\}$. Furthermore, since $T$ leaves $V$ invariant, we have that $T \pi_V=\pi_V T$. Hence, $TP_i=P_i'T$ for $i\in\{1,2,3\}$, where $P_i'=P_{J_i'}=\pi_V J_i'$. Finally,
\begin{align*}
\langle P_i'w, P_j'w\rangle&=\langle P_i'Tv, P_j'Tv\rangle=\langle TP_iv, TP_jv\rangle=\langle P_iv, P_jv\rangle=\cos^2(\varphi_i) \delta_{ij}.
\end{align*}
Since $w$ is arbitrary, by the last claim of Lemma~\ref{lemma:qKa} we get $\Phi(V)=(\varphi_1,\varphi_2,\varphi_3)$.
\end{proof}

\medskip

We now introduce a matrix map that will be very useful in what follows. Let $V$ be a real subspace of $\H^n$ of dimension $k$, and let $\{J_1,J_2,J_3\}$ be a canonical basis of  $\g{J}$. Then, we define  the \textit{K\"ahler angle map} of $V$ with respect to $\{J_1,J_2,J_3\}$ as the map $\Omega$ that sends each unit vector $v\in \mathbb{S}^{k-1}\subset V$ to the symmetric matrix $\Omega(v)$ of order $3$ whose $(i,j)$-entry is given~by
\begin{equation}\label{eq:kahler_angle_map}
\Omega({v}) _{ij}:=\langle P_i v, P_j v \rangle=L_v(J_i,J_j),
\end{equation}
where $P_i=P_{J_i}$, $i\in\{1,2,3\}$.
A straightforward but important observation is that $V$ has constant quaternionic K\"ahler angle if and only if the matrices $\Omega(v)$ have the same eigenvalues counted with multiplicities, for any $v\in \mathbb{S}^{k-1}$. In other words, $\Phi(V)=(\varphi_1,\varphi_2,\varphi_3)$ if and only if the eigenvalues of $\Omega(v)$ are $\cos^2(\varphi_i)$, $i\in\{1,2,3\}$, for all unit $v\in V$. This isospectrality property of the K\"ahler angle map will play a crucial role in this work.

\subsection{Known examples of subspaces with constant quaternionic K\"ahler angle}\label{subsec:known_examples}\hfill

We conclude this section by stating some known partial classifications and examples of subspaces $V$ with constant quaternionic K\"ahler angle in a quaternionic Euclidean space~$\H^n$.

In~\cite{BT07}, Berndt and Tamaru listed some triples that can arise as constant quaternionic K\"ahler angles $\Phi(V)$ of non-zero real subspaces $V$ of $\H^n$, and stated the classification of such particular types of subspaces. All the subspaces in this list are protohomogeneous~\cite{BB01,BT07}.
Such triples are the following:
\begin{enumerate}[{\rm (1)}]
	\item $\Phi(V)=(\pi/2,\pi/2,\pi/2)$. These are precisely the totally real subspaces of $\H^n$. Recall that a linear subspace $V\subset  \H^n$ is totally real if $JV\subset \H^n\ominus V$ for every $J\in\g{J}$. In this case $\dim_\R V\in\{1,2,\dots,n\}$.
	\item $\Phi(V)=(0,\pi/2,\pi/2)$. These are the totally complex subspaces, that is, the subspaces $V$ of $\H^n$ such that $J_1 V\subset V$ and $J V\subset \H^n\ominus V$ for some complex structure $J_1\in\g{J}$ and all $J\in \g{J}$ perpendicular to~$J_1$. In this case $\dim_\R V\in\{2,4,\dots,2n\}$.
	\item $\Phi(V)=(0,0,\pi/2)$. These subspaces are the $3$-dimensional subspaces of the form $\g{J}v=(\mathrm{Im}\, \H)v$ for some non-zero $v\in \H^n$.
	\item $\Phi(V)=(0,0,0)$. These are the quaternionic subspaces, that is, the subspaces $V\subset \H^n$ such that $JV\subset V$ for every $J\in \g{J}$. Hence, $\dim_\R V\in\{4,8,\dots,4n\}$.
	\item $\Phi(V)=(\varphi,\pi/2,\pi/2)$, $\varphi\in(0,\pi/2)$. Let $W$ be a totally complex subspace of $\H^n$, with $J_1 W\subset W$ for some complex structure $J_1\in\g{J}$. Then, a subspace $V$ of $\H^n$ satisfies $\Phi(V)=(\varphi,\pi/2,\pi/2)$ if and only if $V$ is a subspace of some $W$ as before with constant K\"ahler angle $\varphi\in(0,\pi/2)$ as a subspace of the complex vector space $(W,J_1)$. Thus $\dim_\R V\in\{2,4,\dots,2[n/2]\}$.
	\item $\Phi(V)=(0,\varphi,\varphi)$. Let $W$ be a totally complex subspace of $\H^n$ such that $J_2 W\subset W$ for some complex structure $J_2\in \g{J}$, and let $\tilde{V}$ be a real subspace of $(W,J_2)$ with constant K\"ahler angle $\varphi\in(0,\pi/2)$. Then, $V$ is a subspace of $\H^n$ with $\Phi(V)=(0,\varphi,\varphi)$ if and only if it is the complexification $V=J_1 \tilde{V} \oplus \tilde{V}$ of some $\tilde{V}\subset W$ as before with respect to some complex structure $J_1\in \g{J}$ orthogonal to $J_2$.
In this case $\dim_\R V\in\{4,8,\dots,4[n/2]\}$.
\end{enumerate}
We also recall, as observed in~\cite[pp.~3434-3435]{BT07}, that:
\begin{enumerate}[{\rm (i)}]
	\item for each $\ell\in\{1,\dots,n\}$ there exists, up to equivalence, exactly one real subspace $V$ of $\H^n$ with $\dim_\R V$ equal to $\ell$, $2\ell$ or $4\ell$, for each of the types (1), (2) or (4) above, respectively;
	\item there exists only one subspace $V$ of $\H^n$ of type (3), up to equivalence; and
	\item for each $\ell\in\{1,\dots,[n/2]\}$ and each $\varphi \in(0,\pi/2)$ there exists exactly one subspace $V$ of $\H^n$ with $\dim_\R V=2\ell$ of  type (5), and exactly one subspace $V$ of $\H^n$ with $\dim_\R V=4\ell$ of type (6), up to equivalence.
\end{enumerate}

Berndt and Tamaru conjectured in \cite{BT07} that these were all the possible subspaces with constant quaternionic K\"ahler angle, but the first and second authors found in  \cite{damekricci} new examples of  subspaces $V$ of dimension 4 such that $\Phi(V)=(\varphi_1,\varphi_2,\varphi_3)$ where $\cos(\varphi_1) + \cos(\varphi_2)<1+ \cos(\varphi_3)$. These are constructed as follows.
Let $0<\varphi_1\leq \varphi_2\leq\varphi_3\leq\pi/2$ with
$\cos(\varphi_1)+\cos(\varphi_2)<1+\cos(\varphi_3)$, and consider a
$4$-dimensional totally real subspace of~$\H^{n}$ and a basis of
unit vectors $\{e_0,e_1,e_2,e_3\}$ of it, where $\langle
e_0,e_i\rangle=0$, for $i\in\{1,2,3\}$, and
\begin{equation*}\label{eq:inner_products}
\langle e_{i}, e_{i+1}\rangle
=\frac{\cos(\varphi_{i+2})-\cos(\varphi_i)\cos(\varphi_{i+1})}
{\sin(\varphi_i)\sin(\varphi_{i+1})},\quad i\in\{1,2,3\}.
\end{equation*}
For the sake of simplicity let us define $\varphi_0=0$ and $J_0=\Id$.
Notice that $\langle J_j e_k, e_l\rangle=0$ for $j\in\{1,2,3\}$ and
$k$, $l\in\{0,1,2,3\}$, because $\spann_{\R}\{e_0,e_1,e_2,e_3\}$ is
a totally real subspace of~$\H^{n}$. Then we can define
\[
\xi_k = \cos(\varphi_k)J_k e_0+\sin(\varphi_k)J_k e_k, \quad k\in\{0,1,2,3\}.
\]
(Note that $\xi_0=e_0$.) We consider the subspace $V$
spanned by these four vectors, for which $\{\xi_0, \xi_1,
\xi_2,\xi_3\}$ is an orthonormal basis. Then,  $\Phi(V)=(\varphi_1,\varphi_2,\varphi_3)$. It was also observed in~\cite{damekricci} that one can take several copies of these $4$-dimensional subspaces to construct subspaces $V$ of $\H^n$ of dimension multiple of $4$ with $\Phi(V)=(\varphi_1,\varphi_2,\varphi_3)$, where $\cos(\varphi_1)+\cos(\varphi_2)<1+\cos(\varphi_3)$.
This fact will be proved carefully in Section 6 for a broader family of examples that we will provide.

At this point, we find interesting to remark that, unlike the six types of examples known to Berndt and Tamaru in~\cite{BT07}, and as we will see in Proposition~\ref{prop:ineqk=4},
we can prove that, for any positive integer $k$ multiple of $4$, there are triples $(\varphi_1,\varphi_2,\varphi_3)$ for which there are non-equivalent subspaces $V$ of $\H^n$ with $\Phi(V)=(\varphi_1,\varphi_2,\varphi_3)$ and ${\dim_\R V=k}$.

\section{Hairy ball method}\label{sec:hairy-ball}

In this section we use a topological argument to reduce the classification problem of subspaces $V$ with constant quaternionic K\"ahler angle in $\H^n$ to the study of subspaces with dimensions $3$ and multiples of $4$. The idea is to construct a distribution on the unit sphere of the subspace $V$ of $\H^n$, and then use a generalization of the hairy ball theorem to exclude several cases.

Let $V$ be a real subspace of $\H^n$ of real dimension $k$ with constant quaternionic K\"ahler angle $\Phi(V)=(\varphi_1,\varphi_2,\varphi_3)$.
Let $\mathbb{S}^{k-1}$ denote the unit sphere of $V$. For each $v\in \mathbb{S}^{k-1}$ and $J\in\g{J}$ we have $\langle P_J v,v\rangle=0$ and $P_Jv\in V$, and thus $P_Jv\in T_v \mathbb{S}^{k-1}$. For each $v\in \mathbb{S}^{k-1}$ consider the subspace of $T_v \mathbb{S}^{k-1}$ given by
\[
\Delta_v=\{P_Jv:J\in\g{J}\}.
\]
Since $V$ has constant quaternionic K\"ahler angle, the dimension of $\Delta_v$ is independent of $v\in \mathbb{S}^{k-1}$. Hence, $\Delta$ defines a smooth distribution on the sphere $\mathbb{S}^{k-1}$, and its rank coincides with the number of elements $i\in\{1,2,3\}$ such that $\varphi_i\neq \pi/2$.

Steenrod~\cite{Steenrod} computed the possible ranks of continuous distributions on spheres.
We summarize these results in the following statement~\cite[p.~144, Theorem 27.18]{Steenrod}.

\begin{theorem}\label{th:steenrod}
The sphere $\mathbb{S}^\ell$ does not admit a continuous distribution of rank~$r$ if $\ell$ is even and $1\leq r\leq \ell-1$, or if $\ell\equiv 1 \,\mathrm{(mod\; 4)}$ and $2\leq r\leq \ell-2$.
\end{theorem}

Now we can state and prove the main result of this section.
\begin{proposition}
\label{prop:hairyball}
	Let $V$ be a real subspace of $\H^{n}$ with constant quaternionic K\"ahler angle and $\dim_\R V=k$. Then:
	\begin{enumerate}[{\rm (i)}]
		\item  If $k\geq 5$ is odd, then $V$ is a totally real subspace of $\H^{n}$, that is, it has constant quaternionic K\"ahler angle $(\pi/2,\pi/2,\pi/2)$.
		\item If $k\equiv 2 \,\mathrm{(mod\; 4)}$, then $V$ has constant quaternionic K\"ahler angle $(\varphi,\pi/2,\pi/2)$, for some $\varphi\in[0,\pi/2]$.
		\item If $k=3$, then $V$ has constant quaternionic K\"ahler angle $(\varphi, \varphi, \pi/2)$ for some $\varphi\in [0,\pi/2]$.
	\end{enumerate}	
\end{proposition}

\begin{proof}
Let us consider the distribution $\Delta$ defined above in this section. Recall that, by construction, its rank is at most $3$.

Let $k\geq 5$ be odd. Then, Theorem~\ref{th:steenrod} implies that $\mathbb{S}^{k-1}$ does not admit a non-trivial continuous distribution. Thus, the rank of $\Delta$ is $0$. Hence, by definition of $\Delta$ we have $P_Jv=0$ for all $J\in\g{J}$ and $v\in \mathbb{S}^{k-1}$, which means that $\g{J}V$ is perpendicular to $V$. Therefore, $V$ is totally real. This proves (i).

Let now $k\equiv 2 \,\mathrm{(mod\; 4)}$. Theorem~\ref{th:steenrod} guarantees that the rank of $\Delta$ is $0$ or $1$.
If $\Delta$ has rank $1$, then for each $v\in \mathbb{S}^{k-1}$ there is, by definition of $\Delta$, a canonical basis $\{J_1^v,J_2^v,J_3^v\}$ of $\g{J}$ such that $P_1^vv\neq 0$ and $P_{2}^vv=P_3^vv=0$, where $P_i^v=P_{J_i^v}$, for $i\in\{1,2,3\}$. Hence, $v$ has quaternionic K\"ahler angle $(\varphi,\pi/2,\pi/2)$ with respect to $V$ and $\{J_1^v,J_2^v,J_3^v\}$, for some $\varphi\in[0,\pi/2)$. Therefore, $\Phi(V)=(\varphi,\pi/2,\pi/2)$, $\varphi\in[0,\pi/2)$. If $\Delta$ has rank $0$, then $V$ is totally real, as in the proof of  (i). Altogether, we have proved (ii).

Let $k=3$. Then Theorem~\ref{th:steenrod} implies that the rank of $\Delta$ is $0$ or $2$. If it is $0$, then $V$ is totally real. If the rank of $\Delta$ is $2$, then $\Phi(V)=(\varphi_1,\varphi_2,\pi/2)$, for some $\varphi_1,\varphi_2\neq\pi/2$. In this case, let us assume that $\varphi_1 \neq \varphi_2$. Then, for each $v\in \mathbb{S}^2\subset V$ there exist complex structures $J_1^v$ and $J_2^v$ in $\g{J}$, depending continuously on $v$, such that $v$ has K\"ahler angle $\varphi_i\in[0,\pi/2)$ with respect to $J_i^v$ and $V$, for $i\in\{1,2\}$. But then $v\mapsto P_1^vv$ would define a non-vanishing continuous vector field on $\mathbb{S}^2$, which contradicts Theorem~\ref{th:steenrod}. Hence, $\varphi_1=\varphi_2$, which proves (iii).
\end{proof}

In view of Proposition~\ref{prop:hairyball} and the previous partial classification results (\S\ref{subsec:known_examples}), the classification of real subspaces with constant quaternionic K\"ahler angle is reduced to two main cases: subspaces with dimension $k=3$, and subspaces with dimension $k$ multiple of $4$. The case $k=3$ (and hence $\Phi(V)=(\varphi,\varphi,\pi/2)$) will be addressed in~\S\ref{subsec:dim3} by a direct study. The other case is much more involved and, indeed, we will content ourselves with addressing the subcase $k=4$ and, for higher dimensions, restricting our attention to protohomogeneous subspaces. Thus, in Section~\ref{sec:factorization} we will reduce the study of protohomogeneous subspaces of dimension multiple of $4$ to the case of dimension $k=4$, and in \S\ref{subsec:dim4} we will obtain the classification of subspaces of dimension $k=4$ with constant quaternionic K\"ahler angle.

\section{Factorization of subspaces of dimension multiple of four}\label{sec:factorization}

In this section we prove that any protohomogeneous subspace of real dimension $k$ multiple of $4$ in $\H^n$ can be factorized as an $\H$-orthogonal direct sum of subspaces of dimension $4$ with the same constant quaternionic K\"ahler angle. The first step (Subsection~\ref{subsec:canonical}) will be to show, using a Lie group theoretical argument, that the canonical basis of $\g{J}$ provided by Lemma~\ref{lemma:qKa} is independent of the vector in the subspace $V$ of $\H^n$. Then, using this, one can induce a Clifford module structure on $V$, which allows us to conclude the factorization result by using the classification of Clifford modules by Atiyah, Bott and Shapiro~\cite{ABS} (Subsection~\ref{subsec:factorization}).

\subsection{Canonical quaternionic structure}\label{subsec:canonical}\hfill

Let $V$ be a real subspace of a quaternionic Euclidean space $\H^n$. Assume that $V$ is protohomogeneous. Equivalently, $H':=N^0_{\Sp(1)\Sp(n)}(V)$, the connected component of the identity of the normalizer of $V$ in $\Sp(1)\Sp(n)$, acts transitively on the unit sphere $\mathbb{S}^{k-1}$ of~$V$. In particular, $V$ has constant quaternionic K\"ahler angle by Lemma~\ref{lemma:protohomogeneous}.

Consider the subgroup $H''$ of all elements of $H'$ which act trivially on $V$,
\[
H''=Z_{\Sp(1)\Sp(n)}(V)=\left\{h\in H': hv=v,\text{ for all }v\in V\right\}.
\]
This is a closed normal subgroup of $H'$. 
Hence, $H:=H'/H''$ is a compact connected  Lie group. Moreover, the action of $H'$ on $V$ induces an action of $H$ on $V$, 
and the latter inherits the basic properties of the former (it is orthogonal and transitive on the unit sphere $\mathbb{S}^{k-1}$ of $V$), but now the $H$-action is effective. 

The compact connected Lie group $H$ acts effectively and transitively on the unit sphere $\mathbb{S}^{k-1}$ of $V$. Montgomery and Samelson \cite{MS43}, and Borel \cite{Bo50}, classified compact connected Lie groups acting effectively and transitively on spheres (see also \cite[p.~179]{Be87}). In particular (see~\cite[Theorem~I]{MS43}), we have that either $H$ is simple or $H=(H_1\times H_2)/N$, where $H_1$, $H_2$ are connected simple Lie groups and $N$ is a finite normal subgroup of $H_1\times H_2$; moreover, the subgroup of $H$ corresponding to $H_1$ still acts transitively on $\mathbb{S}^{k-1}$.

\begin{proposition}\label{prop:H_subgroup_spn}
Let $V$ be a protohomogeneous real subspace of $\H^n$ of dimension $k\geq 5$. Then, there exists a connected Lie subgroup $S$ of the $\Sp(n)$-factor of $\Sp(1)\Sp(n)$ that acts transitively on the unit sphere $\mathbb{S}^{k-1}$ of $V$. Moreover, the elements of $S$ commute with any complex structure $J\in \g{J}$.
\end{proposition}

\begin{proof}
Let $\g{h}'$ and $\g{h}''$ denote the Lie algebras of $H'$ and $H''$, respectively. Since  $\g{h}'$ is compact, and hence reductive, the ideal $\g{h}''$ of $\g{h}'$ admits a complementary ideal $\g{h}$ of $\g{h}'$ such that $\g{h}'= \g{h}\oplus \g{h}''$ and $\g{h}\simeq \g{h}'/\g{h}''$. Note that the Lie algebra of $H=H'/H''$ is isomorphic to $\g{h}$. If $\widehat{H}$ denotes the connected subgroup of $H'$ with Lie algebra $\g{h}$, then $H'=\widehat{H}\cdot H''$ and hence $H=H'/H''\cong\widehat{H}/(\widehat{H}\cap H'')$ is a finite quotient of $\widehat{H}$.

If $H$ is simple, put $\g{s}:=\g{h}$. If $H$ is not simple, put $\g{s}:=\g{h}_1$, where $\g{h}_1$ is the ideal of $\g{h}$ whose associated connected Lie subgroup of $H$ still acts transitively on $\mathbb{S}^{k-1}$. Note that, in any case, the connected Lie subgroup $S$ of $\widehat{H}\subset H'\subset \Sp(1)\Sp(n)$ with Lie algebra $\g{s}$ acts transitively on the unit sphere $\mathbb{S}^{k-1}$ of $V$.

Recall that $\g{h}'$ is a Lie subalgebra of the direct sum Lie algebra $\g{sp}(1)\oplus\g{sp}(n)$. Consider $\pi_{\g{sp}(1)}\colon \g{sp}(1)\oplus\g{sp}(n)\to\g{sp}(1)$ the projection map onto the first factor, and $\Psi=\pi_{\g{sp}(1)}\rvert_{\g{s}}\colon \g{s}\to\g{sp}(1)$ its restriction to $\g{s}$, which is a Lie algebra homomorphism.

Since $\Ker\Psi$ is an ideal of $\g{s}$ and $\g{s}$ is simple, we have $\Ker\Psi=0$ or $\Ker\Psi=\g{s}$.  If $\Ker\Psi=0$, then $\g{s}$ is isomorphic to a subalgebra of $\g{sp}(1)$;
but $\dim\g{sp}(1)=3$, so $S$ cannot act transitively on $\mathbb{S}^{k-1}$, $k\geq 5$.
Hence, $\Ker\Psi=\g{s}$, and thus, $\Image\Psi=0$, that is,
$\g{s}$ is contained in the $\g{sp}(n)$-factor of the Lie algebra of $\Sp(1)\Sp(n)$.
This proves the first part of the claim.

The connected subgroup $S$ of $\Sp(n)\subset \Sp(1)\Sp(n)$ with Lie algebra $\g{s}$, which acts transitively on the unit sphere of $V$, commutes with the elements of the $\Sp(1)$-factor of $\Sp(1)\Sp(n)$. Since the quaternionic structure $\g{J}$ of $\H^n$ is induced precisely by the action of the $\Sp(1)$-factor on $\H^n$, we obtain that the elements of $S$ commute with any $J\in \g{J}$.
\end{proof}

As a consequence, we have

\begin{corollary}\label{prop:in_spn}
Let $V\subset \H^n$ be a protohomogeneous real subspace of dimension $k\geq5$ with constant quaternionic K\"ahler angle $\Phi(V)=(\varphi_1,\varphi_2,\varphi_3)$. Then, there exists a canonical basis $\{J_1,J_2,J_3\}$ of $\mathfrak{J}$ such that the K\"ahler angle of any unit vector $v\in V$ with respect to $J_i$ and $V$ is $\varphi_i$, for each $i\in\{1,2,3\}$.
\end{corollary}

\begin{proof}
It suffices to show that the bilinear form $L_v$ given in Lemma~\ref{lemma:qKa} is independent of $v\in \mathbb{S}^{k-1}$. Indeed, given $v,w\in \mathbb{S}^{k-1}$, there exists $T\in S$ such that $Tv=w$. Since $T$ commutes with all $J\in\g{J}$ and preserves $V$, we have
\begin{align*}
L_w(J,J')&=\langle P_J w, P_{J'}w\rangle =
\langle \pi_V J Tv, \pi_V J'Tv\rangle =
\langle \pi_V T J v, \pi_V T J'v\rangle
\\
&=\langle T \pi_V J v,  T\pi_V J'v\rangle  =
\langle T P_J v, T P_{J'}v\rangle =\langle P_J v, P_{J'}v\rangle=L_v(J,J'),
\end{align*}
for all $J$, $J'\in\g{J}$.
\end{proof}

\subsection{Factorization Lemma}\label{subsec:factorization}\hfill

Let $V$ be a real subspace of $\H^n$ of constant quaternionic K\"ahler angle $(\varphi_1,\varphi_2,\varphi_3)$ with $\varphi_2\neq \pi/2$. Assume that there exists a canonical basis $\{J_1,J_2,J_3\}$ of $\mathfrak{J}$ such that the K\"ahler angle of any non-zero vector $v\in  V$ with respect to $J_i$ and $V$ is $\varphi_i$, for $i\in\{1,2,3\}$.  Note that, by Corollary~\ref{prop:in_spn}, if $V$ is protohomogeneous of dimension at least $5$, then the previous assumption holds. In view of Proposition~\ref{prop:hairyball} (and leaving the case $k=3$ for later), we will assume that $\dim_\R V=4l$ with $l\in \mathbb{N}$.

Let us regard $\mathbb{H}^n$ as a complex vector space $\mathbb{C}^{2n}$ with respect to the complex structure $J_i$. By \cite[p.~1191]{mathz}, we have that $\bar{P}_{i}:=P_i/\cos(\varphi_i)=\pi_VJ_i/\cos(\varphi_i)$ leaves $V$ invariant and defines an orthogonal complex structure in $V$, for each $i\in\{1,2\}$, and also for $i=3$ if and only if $\varphi_3\neq\pi/2$. Furthermore, we can easily check that $\bar{P}_{i}\bar{P}_{j}=-\bar{P}_{j}\bar{P}_{i}$, for $i\neq j$.
Indeed, if $v,w \in V$, then Lemma~\ref{lemma:qKa} yields
\begin{align*}
0&=\langle\bar{P}_{i}(v+w),\bar{P}_{j}(v+w)\rangle
=\langle\bar{P}_{i}v,\bar{P}_{j}w \rangle + \langle\bar{P}_{j}v,\bar{P}_{i}w  \rangle=-\langle\bar{P}_{j}\bar{P}_{i}v,w \rangle - \langle\bar{P}_{i}\bar{P}_{j}v,w\rangle.
\end{align*}
Hence, $V$ has a module structure over the Clifford algebra $\mathsf{Cl}(3)$ if $\varphi_3\neq \pi/2$, or over $\mathsf{Cl}(2)$ if $\varphi_3=\pi/2$. It is well known that there are exactly two inequivalent irreducible Clifford modules over $\mathsf{Cl}(3)$, both of dimension $4$ (we will denote them by $V^0$ and $V^1$), whereas there is exactly one irreducible $\mathsf{Cl}(2)$-module up to equivalence, again of dimension $4$ (we will denote it by $V^0$). Moreover, Clifford modules are semisimple. This implies that, if $\varphi_3\neq\pi/2$, we can decompose $V$ into a direct sum of irreducible $\mathsf{Cl}(3)$-modules as follows
\[
V=\left(\bigoplus^{l_0} V^0\right) \oplus \left(\bigoplus^{l_1} V^1\right),
\]
where $l_0 + l_1=l$, whereas if $\varphi_3=\pi/2$ the $\mathsf{Cl}(2)$-module $V$ can be decomposed as
\[
V=\bigoplus_{}^l V^0.
\]
The above decompositions can be assumed to be orthogonal because the complex structures $\bar{P}_i$ are orthogonal.
This also implies that two different summands are $\H$-orthogonal:
if $v,w\in V$ belong to two different summands, $\langle J_k v,w\rangle=\langle P_k v,w\rangle=0$ by the $\bar{P}_k$-invariance.
Finally, since $\bar{P}_{i}$ leaves each factor $V^r$ ($r\in\{0,1\}$) invariant, we deduce that each $V^r$ has constant quaternionic K\"ahler angle $(\varphi_1,\varphi_2,\varphi_3)$. This leads us to state the following:

\begin{lemma}\label{lemma:factorization}
	Let $V$ be a real subspace of $\mathbb{H}^n$ of dimension $4l$, with $l\in \mathbb{N}$, and constant quaternionic K\"ahler angle  $(\varphi_1,\varphi_2,\varphi_3)$. Assume that there exists a canonical basis $\{J_1,J_2,J_3\}$ of $\g{J}$ such that the K\"ahler angle of any non-zero $v\in V$ with respect to $J_i$ and $V$ is $\varphi_i$, for each $i\in\{1,2,3\}$.  Then, there is an $\H$-orthogonal decomposition
	\[
    V = \bigoplus_{r=1}^l V_r,
    \]
	where each $V_r$ has dimension $4$ and  $\Phi(V_r)=(\varphi_1,\varphi_2,\varphi_3)$ as a subspace of $\mathbb{H}^n$.
	
	Conversely,	let $V$ be a real subspace of $\H^n$ given by an $\H$-orthogonal direct sum $V:=\bigoplus_{r=1}^l V_r$, where each $V_r$ has dimension $4$, and $\Phi(V_r)=(\varphi_1,\varphi_2,\varphi_3)$. Let $\{J_1,J_2,J_3\}$  be a canonical structure of $\g{J}$ such that every non-zero vector in $V_r$ has K\"ahler angle $\varphi_i$ with respect to $J_i$ and $V_r$, for each $i\in\{1,2,3\}$ and each $r\in\{1,\dots, l\}$. Then, $\Phi(V)=(\varphi_1,\varphi_2,\varphi_3)$.
\end{lemma}

\begin{proof}
	The first assertion has been proved above under the assumption $\varphi_2\neq\pi/2$. If $\varphi_2=\pi/2$, the first claim follows from the classification of subspaces $V$ with $\Phi(V)=(\varphi,\pi/2,\pi/2)$, $\varphi\in[0,\pi/2]$ (cf.~\S\ref{subsec:known_examples} and~\cite[pp.~230-232]{BB01}).
	
	In order to prove the converse, we first note that $\pi_V(\H V_r)=V_r$, for each $r\in\{1,\dots,l\}$. Indeed, for every $v\in V_r$ and $w\in V_s$, $r\neq s$, $\langle \pi_V J_i v, w\rangle=\langle J_i v,w\rangle=0$, where in the last equality we have used $\H V_r \perp\H V_s$. Hence, $\pi_V \g{J} (V_r)\subset V_r$, and since $\pi_V(V_r)=V_r$, we deduce $\pi_V(\H V_r)=V_r$.
	
	Now let $v=\sum_{r=1}^l v_r\in V$, with $v_r\in V_r$ for each $r\in\{1,\dots, l\}$. Denoting as usual $P_i=\pi_V J_i$, for each $i\in\{1,2,3\}$, we have
\begin{align*}
L_v(J_i,J_j)
&{}=\langle P_i v, P_j v\rangle
=\sum_{r,s=1}^l\langle \pi_V J_iv_r,\pi_V J_jv_s\rangle
=\sum_{r=1}^l\langle \pi_V J_i v_r,\pi_V J_j v_r\rangle
\\
&{}=\sum_{r=1}^l\langle P_i v_r,P_j v_r\rangle
=\sum_{r=1}^l\cos^2(\varphi_i)\delta_{ij} \lVert v_r\rVert^2
=\cos^2(\varphi_i)\delta_{ij}\lVert v\rVert^2,
\end{align*}
where in the third equality we have used $\pi_V(\H V_r)=V_r$ and $V_r\perp V_s$ for all $r,s\in\{1,\dots, l\}$, and in the fifth one we have used that the quaternionic K\"{a}hler angle of $v_r$ with respect to $V_r$ and $\{J_1,J_2,J_3\}$ is $(\varphi_1,\varphi_2,\varphi_3)$. 
Since $v\in V$ is arbitrary, by Lemma~\ref{lemma:qKa} we conclude that $\Phi(V)=(\varphi_1,\varphi_2,\varphi_3)$.
\end{proof}

\section{Low dimensional subspaces with constant quaternionic K\"ahler angle}\label{sec:low-dim}

As a consequence of Proposition~\ref{prop:hairyball}, we only have to study subspaces  of dimensions $3$ and multiples of $4$.
The latter can be reduced to studying subspaces of dimension $4$ by virtue of Corollary~\ref{prop:in_spn} and Lemma \ref{lemma:factorization}. We will devote this section to the classification of (not necessarily protohomogeneous)  real subspaces of dimensions $k\in\{3,4\}$ with constant quaternionic K\"ahler angle. The main tool that we will use in this section is the isospectrality of the K\"ahler angle map $\Omega$ introduced in~\eqref{eq:kahler_angle_map}.
We start with a lemma that provides an appropriate basis of the subspace.

\begin{lemma}\label{lemma:basisk34}
Let $V$ be a real subspace of $\H^n$ of dimension $k\in\{3,4\}$ with $\Phi(V)=(\varphi_1,\varphi_2,\varphi_3)$. Let $e_0\in V$ be a unit vector.
Then, there exists a canonical basis $\{J_1,J_2,J_3\}$ of $\g{J}$ and vectors $e_i\in \H^n\ominus \H e_0$, $i\in\{1,\dots,k-1\}$, such that
\begin{equation}\label{eq:basisk34}
\cos(\varphi_i)J_i e_0 + \sin(\varphi_i)J_i e_i,\qquad i\in\{0,\dots, k-1\},
\end{equation}
constitute an $\R$-orthonormal basis of~$V$, where we put $J_0:=\Id$ and $\varphi_0=0$.

Moreover, for each $i\in\{0,\dots,k-1\}$ with $\varphi_i\neq \pi/2$, we have 
\[
\bar{P}_ie_0=\cos(\varphi_i)J_i e_0 + \sin(\varphi_i)J_i e_i,
\]
where $\bar{P}_i=P_i/\cos(\varphi_i)=\pi_V J_i/\cos(\varphi_i)$.

Finally, if $\varphi_i=0$ we take $e_i=0$, whereas if $\varphi_i>0$, then $e_i$ is a unit vector.
\end{lemma}

\begin{proof}
Let $e_0\in V$ be a unit vector. By Lemma \ref{lemma:qKa}, there is a canonical basis $\{J_1,J_2,J_3\}$ of $\mathfrak{J}$ such that $e_0$ has K\"ahler angle $\varphi_i$ with respect to $J_i$ for $i\in\{1,2,3\}$, and $\langle P_i e_0,P_j e_0\rangle = \cos^2(\varphi_i)\delta_{ij}$.
In particular, $\langle \bar{P}_i e_0, \bar{P}_j e_0\rangle=0$ for any $i$, $j\in\{0,\dots, k-1\}$, $i\neq j$, with $\varphi_i$, $\varphi_j\neq \pi/2$.

Fix $i\in\{1,2,3\}$. If $\varphi_i=0$, then we take $e_i=0$. Let us assume first that $\varphi_i\in (0,\pi/2)$. By regarding $\H^n$ as a complex vector space $\C^{2n}$ with respect to the complex structure $J_i$, \cite[Lemma~2]{BB01} yields the existence of a unit vector $e_i\in \H^n \ominus \spann_{\R}\{e_0, J_i e_0\}$ satisfying
\[
\bar{P}_i e_0=\cos(\varphi_i) J_i e_0 + \sin(\varphi_i) J_i e_i.
\]
We have to see that $e_i\in \H^n \ominus \H e_0$.
Observe that $\H^n \ominus \spann_{\R}\{e_0, J_i e_0\}$ coincides with the orthogonal sum $(\H^n \ominus \H e_0)\oplus\spann_{\R}\{J_{i+1} e_0, J_{i+2}e_0\}$, where indices are taken modulo~$3$.
Let $a,b \in \R$. Then
\begin{align*}
\langle e_i, \,a J_{i+1}e_0 + b J_{i+2} e_0\rangle
&{}=-\frac{1}{\sin(\varphi_i)}\langle J_i\bar{P}_{i} e_0 + \cos(\varphi_i)e_0,
\,a J_{i+1}e_0+b J_{i+2}e_0\rangle\\
&{}=\frac{1}{\sin(\varphi_i)\cos(\varphi_i)}\left(a\langle P_i e_0, P_{i+2}e_0\rangle
- b \langle P_i e_0, P_{i+1}e_0\rangle\right)=0,
\end{align*}
where in the last equality we have used Lemma \ref{lemma:qKa}.
Therefore, $e_i\in \H^n \ominus \H e_0$.

Now if $\varphi_2=\pi/2$, subspaces $V$ with $\Phi(V)=(\varphi,\pi/2,\pi/2)$, $\varphi\in[0,\pi/2]$, are classified (see~\S\ref{subsec:known_examples}) and they can be spanned by a basis as in the statement (see \cite[p.~232]{BB01} and note that the $\{e_i\}$ in the statement do not have to be $\H$-orthonormal).

Thus, we finally have to deal with the case $k=4$, $\varphi_2\neq \pi/2$, and $\varphi_3=\pi/2$. Then, by the previous argument, there exists a unit vector $v\in \H^n$ such that $\{e_0,\bar{P}_1e_0,\bar{P}_2e_0,v\}$ is an $\R$-orthonormal basis of~$V$, where $\bar{P}_ie_0=\cos(\varphi_i)J_i e_0 + \sin(\varphi_i)J_i e_i$, $i\in\{1,2\}$.
Recalling the definition of the K\"ahler angle map~\eqref{eq:kahler_angle_map}, we have
\begin{align*}
\tr(\Omega(e_0))&=\sum_{i=1}^3 \langle P_i e_0, P_i e_0\rangle
=\sum_{i=1}^3 \left(\langle P_ie_0,e_0\rangle^2 +\langle P_ie_0,\bar{P}_1e_0\rangle^2+\langle P_ie_0,\bar{P}_2e_0\rangle^2+ \langle P_i e_0,v\rangle^2\right)\\
&= \cos^2(\varphi_1)+\cos^2(\varphi_2)+\sum_{i=1}^3  \langle J_i e_0,v\rangle^2 ,
\end{align*}
where we have used Lemma \ref{lemma:qKa} and $\bar{P}_ie_0=P_ie_0/\cos(\varphi_i)$.
Since $\Phi(V)=(\varphi_1,\varphi_2,\pi/2)$, the eigenvalues of $\Omega(e_0)$ are $\cos^2(\varphi_1)$, $\cos^2(\varphi_2)$ and $0$, and hence we deduce that $v\in \H^n\ominus\H e_0$.
Thus, taking $e_3=-J_3 v$ yields the result.
\end{proof}
\begin{remark}
\label{rem:ortho}
Whenever $\varphi_1>0$, the orthogonality of (\ref{eq:basisk34}) yields $\langle J_3 e_1, e_2\rangle=0$, and if $k=4$, also $\langle J_1 e_2, e_3\rangle=\langle J_2 e_3, e_1\rangle=0$.
\end{remark}

\subsection{Subspaces of dimension $3$}\label{subsec:dim3}\hfill

In this subsection we classify $3$-dimensional real subspaces of $\H^n$ with constant quaternionic K\"ahler angle.

\begin{proposition}\label{prop:sk=3}
	Let $V\subset \H^n$ be a real subspace of  dimension $3$. Then, $V$ has constant quaternionic K\"ahler angle if and only if $\Phi(V)=(\varphi,\varphi,\pi/2)$, $\varphi\in[0,\pi/2]$, and for any unit $e_0\in V$, there is a canonical basis $\{J_1,J_2,J_3\}$ of $\g{J}$ such that
	\begin{equation}\label{eq:basis_k3}
	\{e_0, \cos(\varphi) J_1 e_0 + \sin(\varphi) J_1 e_1,\cos(\varphi) J_2 e_0 + \sin(\varphi) J_2 e_2\}
	\end{equation}
	is an orthonormal basis of $V$, where, if $\varphi\neq 0$, $e_1, e_2$ are unit vectors satisfying $e_1,e_2\in\H^n\ominus\H e_0$, $e_2\in\H^n\ominus(\mathrm{Im}\, \H)e_1$, and either $\langle e_1, e_2 \rangle= \cos(\varphi)/(\cos(\varphi)-1)$ with $\varphi\in[\pi/3,\pi/2]$, or $\langle e_1, e_2 \rangle= \cos(\varphi)/(\cos(\varphi)+1)$  with $\varphi\in(0,\pi/2]$.
\end{proposition}

\begin{proof}
	By Proposition \ref{prop:hairyball}, we have that $\varphi_1=\varphi_2=\varphi\in[0,\pi/2]$ and $\varphi_3=\pi/2$. Let us assume that $V$ is spanned by the basis described in Lemma \ref{lemma:basisk34} with $k=3$. If $\varphi=0$ or $\varphi=\pi/2$ the claim follows from the classification of subspaces with constant quaternionic K\"ahler angle $(0,0,\pi/2)$ or $(\pi/2,\pi/2,\pi/2)$; see \S \ref{subsec:known_examples}.
	
	Thus, let us assume $\varphi\in(0,\pi/2)$. Then, for each $l\in\{1,2\}$ and understanding the subscript $l+1\in\{1,2\}$ modulo $2$,
	\begin{align*}
	\Omega(\bar{P}_l e_0)_{ij}
	&{}=\langle  P_i \bar{P}_l e_0, P_j \bar{P}_l e_0\rangle
	=\langle J_i \bar{P}_le_0,e_0\rangle \langle J_j\bar{P}_le_0,e_0\rangle
    + \sum_{r=1}^2 \langle J_i \bar{P}_l e_0, \bar{P}_r e_0\rangle
    \langle J_j \bar{P}_l e_0,\bar{P}_r e_0\rangle,
	\\
	&{}=\langle \bar{P}_l e_0, P_i e_0\rangle\langle \bar{P}_l e_0,P_j e_0\rangle
    + \langle J_i \bar{P}_l e_0, \bar{P}_{l+1} e_0\rangle
    \langle J_j \bar{P}_l e_0,\bar{P}_{l+1} e_0\rangle,
	\end{align*}
	where in the second equality we have calculated the orthogonal projection of vectors onto $V$ by using  the orthonormal basis $\{e_0,\bar{P}_1e_0,\bar{P}_2e_0\}$ of~$V$. Hence, for $l\in\{1,2\}$, using Lemma~\ref{lemma:basisk34} we have
\begin{equation}\label{eq:omegaij3}
\begin{aligned}	
\Omega(\bar{P}_l e_0)_{ll}
&{}=\cos^2(\varphi) + \langle e_l, J_{l+1} e_{l+1} \rangle^2 \sin^4(\varphi), \\
\Omega(\bar{P}_l e_0)_{l+1,l}
&{}=\langle e_{1}, J_{1} e_{2} \rangle\langle e_{1}, J_{2} e_{2} \rangle \sin^4(\varphi),\\
\Omega(\bar{P}_l e_0)_{l+1,l+1}
&{}= \langle e_{l+1}, J_{l} e_{l} \rangle^2 \sin^4(\varphi), \\
\Omega(\bar{P}_l e_0)_{13}
&{}=\langle e_2, J_2 e_1\rangle \sin^2(\varphi)(\cos^2(\varphi) + \langle e_1, e_2\rangle\sin^2(\varphi)),\\
\Omega(\bar{P}_l e_0)_{23}
&{}=-\langle e_1, J_1 e_2\rangle \sin^2(\varphi)(\cos^2(\varphi) + \langle e_1, e_2\rangle\sin^2(\varphi)),\\
\Omega(\bar{P}_l e_0)_{33}
&{}=(\cos^2(\varphi) + \langle e_1, e_2\rangle\sin^2(\varphi))^2.
\end{aligned}
\end{equation}
	Now, since $\Omega(\bar{P}_l e_0)$ is symmetric with eigenvalues $\cos^2(\varphi)$ (of multiplicity $2$) and $0$, by the min-max theorem, one obtains
	\[
    0\leq \Omega(\bar{P}_l e_0)_{ll}  \leq \cos^2(\varphi), \qquad l\in\{1,2\}.
    \]
	This implies $\langle e_1, J_2 e_2\rangle =\langle e_2, J_1 e_1\rangle=0$, which together with Remark~\ref{rem:ortho} yields $e_2\in\H^n\ominus(\mathrm{Im}\, \H)e_1$. 
Taking again into account the spectrum of $\Omega(\bar{P}_1 e_0)$, we have the following relation for its trace,
	\[2\cos(\varphi)^2=\tr(\Omega(\bar{P}_1 e_0))=(\cos^2(\varphi) + \langle e_1, e_2\rangle\sin^2(\varphi))^2+\cos ^2(\varphi ).\]
From this and the fact that $e_1$ and $e_2$ are unit vectors, we deduce that either $\langle e_1, e_2 \rangle= \cos(\varphi)/(\cos(\varphi)-1)$ where $\varphi\in[\pi/3,\pi/2)$ or $\langle e_1, e_2 \rangle= \cos(\varphi)/(1+\cos(\varphi))$  where $\varphi\in(0,\pi/2)$. This proves the necessity of the statement.

For the converse we take an arbitrary unit vector $v\in V$ which we write as
\[
v=x_0 e_0 + x_1 \bigl(\cos(\varphi)J_1 e_0+\sin(\varphi)J_1 e_1\bigr)
+ x_2 \bigl(\cos(\varphi)J_2 e_0+\sin(\varphi)J_2 e_2\bigr).
\]
Then, if $\varepsilon\in\{\pm 1\}$ is such that $\langle e_1,e_2\rangle =\cos(\varphi)/(1+\varepsilon\cos(\varphi))$, we have
\[
\Omega(v)=\cos^2(\varphi)
\begin{pmatrix}
x_0^2+x_1^2  &   
x_1 x_2 &   
-\varepsilon x_0 x_2
\\
x_1 x_2 &   
x_0^2+x_2^2  &   
\varepsilon x_0 x_1
\\
-\varepsilon x_0 x_2  &   
\varepsilon x_0 x_1  &   
x_1^2+x_2^2
\end{pmatrix}.
\]
Since $v$ is a unit vector, $x_0^2+x_1^2+x_2^2=1$, and it is now easy to see that $\Omega(v)$ has a double eigenvalue $\cos^2(\varphi)$, and a simple eigenvalue $0$.
\end{proof}

\begin{remark}\label{rem:k3smaller_n}
	We will denote by $V_{+}^\varphi$ and $V_{-}^\varphi$ any real subspace of $\H^n$ constructed as in Proposition~\ref{prop:sk=3},
	depending on whether  $\langle e_1, e_2 \rangle= \cos(\varphi)/(\cos(\varphi)+1)$ for $\varphi\in(0,\pi/2]$, or  $\langle e_1, e_2 \rangle= \cos(\varphi)/(\cos(\varphi)-1)$ for $\varphi\in[\pi/3,\pi/2]$, respectively.
	Note that the subspaces $V_\pm^\varphi$ can be constructed as subspaces of any $\H^n$ with $n\geq 3$. One can easily check that the only one that fits into an $\H^2$ is $V_-^{\pi/3}$ (but it cannot fit into $\H$).
\end{remark}

\begin{proposition}
	\label{proto3}
	Let $V$ be a subspace of $\H^n$ with constant quaternionic K\"ahler angle and dimension $3$. Then $V$ is protohomogeneous.
\end{proposition}

\begin{proof}
	We know from Proposition~\ref{prop:sk=3} that $\Phi(V)=(\varphi,\varphi,\pi/2)$. We can assume that $\varphi\in (0,\pi/2)$ since, otherwise, $V$ is known to be protohomogeneous (see~\S\ref{subsec:known_examples}).
	
	Let $e_0\in V$ be an arbitrary unit vector. By Lemma \ref{lemma:qKa} there is a canonical basis $\{J_1,J_2,J_3\}$ of $\g{J}$ such that $e_0$ has K\"ahler angle $\varphi$ with respect to $J_1$ and $J_2$, and K\"ahler angle $\pi/2$ with respect to $J_3$. In view of Lemma~\ref{lemma:basisk34} and Proposition~\ref{prop:sk=3}, let us consider the unit vectors $e_i\in \H^n \ominus  \H e_0$, $i\in\{1,2\}$, given by
	\begin{equation}\label{eq:e_i}
	e_i:=-(J_i \bar{P}_i e_0 + \cos(\varphi)e_0)/\sin(\varphi), \qquad i\in\{1,2\},
	\end{equation}
	where $\bar{P}_i:=\pi_{V} J_i/\cos(\varphi)$. On the one hand, by~\eqref{eq:e_i} we have
	\begin{equation}
	\label{eq:prodf}
	\begin{aligned}
	\langle e_1, e_2 \rangle&= \frac{1}{\sin^2(\varphi)}\langle J_1 \bar{P}_1 e_0 + \cos(\varphi) e_0, J_2 \bar{P}_2 e_0 + \cos(\varphi) e_0\rangle\\
	&=  \frac{1}{\sin^2(\varphi)}\bigl( \langle J_1 \bar{P}_1 e_0, J_2 \bar{P}_2 e_0\rangle - \cos^2(\varphi)\bigr).
	\end{aligned}
	\end{equation}
	On the other hand, again by Proposition~\ref{prop:sk=3}, $\langle e_1,e_2\rangle$ can take two possible values. We will first see that, given $V$, $\langle e_1,e_2\rangle$ is independent of $e_0$.
	
Let $\mathbb{S}^2$ denote the unit sphere of $V$.
We define $\Theta \colon \mathbb{S}^2\to \mathbb{R}$  by $\Theta(e_0)=\langle e_1, e_2\rangle$. We claim that $\Theta$ is well defined.
	Let $\{J'_1, J'_2, J_3\}$ be another canonical basis of $\mathfrak{J}$ such that $e_0$ has K\"ahler angle $\varphi$ with respect to $J'_i$, $i\in\{1,2\}$, and let $e'_i:=-(J'_i \bar{P}'_i e_0 + \cos(\varphi) e_0)/\sin(\varphi)$ where $\bar{P}'_i:=\pi_V  J'_i /\cos(\varphi)$ for $i\in\{1,2\}$. Then, there is $\theta\in[0,2\pi)$ such that $J'_i=\cos(\theta) J_i+ (-1)^{i+1} \sin(\theta) J_{i+1}$ for $i\in\{1,2\}$ and subscripts modulo $2$. Thus,
	\begin{equation}
	\label{eq:composj}
	\begin{aligned}
	J'_i \bar{P}'_i&=(\cos(\theta) J_i + (-1)^{i+1} \sin(\theta) J_{i+1})(\cos(\theta)\bar{P}_i + (-1)^{i+1} \sin(\theta) \bar{P}_{i+1})\\
	&=\cos^2(\theta) J_i \bar{P}_i + \sin^2(\theta)J_{i+1} \bar{P}_{i+1} +(-1)^{i+1}\cos(\theta)\sin(\theta)(J_1 \bar{P}_2 + J_2 \bar{P}_1).
	\end{aligned}
	\end{equation}
	Consequently, using Equation (\ref{eq:prodf}) twice, and then (\ref{eq:composj}), we get,  after some calculations,
\[
\langle e_1, e_2\rangle -\langle e'_1, e'_2\rangle=\frac{1}{\sin^2(\varphi)}\left(\langle J_1 \bar{P}_1 e_0, J_2 \bar{P}_2 e_0 \rangle -  \langle J'_1 \bar{P}'_1 e_0, J'_2 \bar{P}'_2 e_0 \rangle \right)=0,
\]
	which implies that $\Theta$ is well-defined.
	
	Now note that the assignment $e_0\in \mathbb{S}^2\mapsto \spann\{J_1,J_2\}\in G_2(\g{J})$, where $G_2(\g{J})$ is the Grassmannian of $2$-planes of $\g{J}\cong\R^3$, is continuous due to the continuous dependence of the quadratic form $J\in\g{J}\mapsto L_v(J,J)=\langle P_Jv,P_Jv\rangle\in\R$ on $v$. Hence, the map $\Theta$ is also continuous. But, as mentioned just after~\eqref{eq:prodf}, $\Theta(\mathbb{S}^2)$ has at most two elements. Therefore, $\Theta$ is constant on~$\mathbb{S}^2$.
	
Finally, we prove that $V$ is protohomogeneous. Let $e_0$, $e_0'$ be arbitrary unit vectors in~$V$. Let $\{J_1,J_2,J_3\}$, $\{J_1',J_2',J_3'\}$ be canonical bases of $\g{J}$, and $e_1,e_2,e_1',e_2'$ be unit vectors in $V$ such that both \eqref{eq:basis_k3}, and $\eqref{eq:basis_k3}$ with $e_i'$ instead of $e_i$ and $J_i'$ instead of $J_i$, are orthonormal bases of $V$.	
Both sets of vectors $\{e_0,e_1,e_2\}$ and $\{e_0',e_1',e_2'\}$ span a totally real subspace of $\H^n$, and since $\Theta$ is constant, $\langle e_i,e_j\rangle=\langle e_i',e_j'\rangle$ for all $i,j\in\{0,1,2\}$.
It then follows that there exists an element $T\in \Sp(1)\Sp(n)$ such that $Te_i=e_i'$ for each	 $i\in\{0,1,2\}$, and $TJ_j=J'_jT$ for each $j\in\{1,2,3\}$.  Thus, by~\eqref{eq:e_i} we get $T\bar{P}_i e_0=\bar{P}'_i e_0'$ for
	$i\in\{0,1,2\}$, where $\bar{P}_0=\bar{P}'_0=\Id$. Therefore, $T$ is an element of $\Sp(1)\Sp(n)$ such that $TV=V$ and $Te_0=e_0'$. Since $e_0,e_0'$ are arbitrary, this proves that $V$ is protohomogeneous.
\end{proof}

Finally we show that the two types of subspaces $V_+^\varphi$ and $V_-^\varphi$ introduced in Remark~\ref{rem:k3smaller_n}
are indeed inequivalent for $\varphi\neq\pi/2$. Recall that $V_+^\varphi$ is defined for all $\varphi\in(0,\pi/2]$, but $V_-^\varphi$ only for $\varphi\in[\pi/3,\pi/2]$.

\begin{proposition}\label{prop:ineqk=3}
	Let $\varphi\in[\pi/3,\pi/2]$. Then there exists $T\in \Sp(1)\Sp(n)$ such that $TV_+^\varphi=V_{-}^\varphi$ if and only if $\varphi=\pi/2$.
\end{proposition}

\begin{proof}
	If $\varphi=\pi/2$, then $V_+^{\pi/2}$ and $V_-^{\pi/2}$ are totally real, therefore equivalent. Let us assume that $\varphi\neq \pi/2$ and that there is $T\in \Sp(1)\Sp(n)$ such that $T V_+^\varphi= V_{-}^\varphi$.
By applying an element of $\Sp(1)\Sp(n)$ if necessary, we can assume that there is a unit vector $e_0\in V_+^\varphi\cap V_{-}^\varphi$ and that $e_0$ has quaternionic K\"ahler angle $(\varphi,\varphi,\pi/2)$ with respect to both $V^\varphi_+$ and $V^\varphi_-$ and a common canonical basis $\{J_1,J_2,J_3\}$ of $\g{J}$.
	Then, by Lemma~\ref{lemma:basisk34} and Proposition~\ref{prop:sk=3}, $V_\pm^\varphi$ is spanned by the basis $\{e_0,\bar{P}_{1}^\pm e_0,\bar{P}_{2}^\pm e_0\}$, 
	where  $\bar{P}_{i}^\pm:=\pi_{V_\pm^\varphi}  J_i /\cos(\varphi)$, $i\in\{1,2\}$. Moreover, 
\[
\bar{P}^\pm_i e_0=\cos(\varphi)J_i e_0 + \sin(\varphi) J_i e^\pm_i, \,\text{ with }
\langle e_1^\pm,e_2^\pm\rangle=\frac{\cos(\varphi)}{\cos(\varphi)\pm 1},
\]
and $e_i^\pm\in\H^n\ominus\H e_0$,  $i\in\{1,2\}$.
	
	By Proposition \ref{proto3}, we can assume that $T e_0=e_0$. Let $J'=TJ_3T^{-1}\in\g{J}$. Since $J_3e_0\in\H^n\ominus V_+^\varphi$, we have $J'e_0=J'Te_0=TJ_3e_0\in \H^n \ominus V_{-}^\varphi$.
	This implies $TJ_3=\varepsilon J_3 T$, where $\varepsilon\in\{-1,1\}$, because $\pm J_3$ are the only complex structures in $\mathfrak{J}$ which send $e_0$ to $\H^n \ominus V_{-}^\varphi$.
	Therefore, there exists $\theta\in[0,2\pi)$ such that
	\begin{equation}\label{eq:conj_by_T}
	TJ_i=\varepsilon^i(\cos(\theta) J_i + (-1)^{i+1} \sin(\theta) J_{i+1})T, \quad i\in\{1,2\},\quad \text{ and } \quad TJ_3=\varepsilon J_3 T.
	\end{equation}
	Using~\eqref{eq:conj_by_T} and $Te_0=e_0$, we have
\begin{equation}\label{eq:TP_1+}
\begin{aligned}
T\bar{P}^{+}_1 e_0
&{}=\cos(\varphi) T J_1 e_0 + \sin(\varphi) T J_1  e^+_1
\\
&{}=\varepsilon\cos(\varphi)(\cos(\theta) J_1 e_0 + \sin(\theta) J_2  e_0) + \varepsilon\sin(\varphi)(\cos(\theta)J_1Te_1^+ + \sin(\theta) J_2 Te^+_1).
\end{aligned}
\end{equation}
	
	By Proposition~\ref{proto3}, $V_\pm^\varphi$ is protohomogeneous, and note that $\SO(3)$ is the only connected subgroup of $\Sp(1)\Sp(n)\subset \SO(4n)$ that acts transitively and effectively on the unit sphere of $V_\pm^\varphi$. Thus, we can assume that $T\bar{P}_{1}^{+}e_0=\varepsilon \bar{P}_{1}^{-}e_0$, just by composing $T$ with some element in the isotropy of the action of $\SO(3)$ on $V_-^\varphi$ at $e_0$. But inserting~\eqref{eq:TP_1+} and $\bar{P}^{-}_1 e_0=\cos(\varphi)  J_1 e_0 + \sin(\varphi)J_1  e^-_1$ into the equality $T\bar{P}_{1}^{+}e_0=\varepsilon \bar{P}_{1}^{-}e_0$, and analyzing the $\H e_0$ and $\H^n\ominus\H e_0$ components (note that $e_1^\pm\in\H^n\ominus\H e_0$, $Te_0=e_0$, and $T$ preserves $\H$-orthonormality) we get $\theta=0$ and $Te_1^+=e_1^-$. Moreover, by~\eqref{eq:conj_by_T} we get $TJ_i=\varepsilon^i J_i T$,  $i\in\{1,2,3\}$.
	
	Since $Te_0=e_0$, $T\bar{P}_{1}^{+}e_0=\varepsilon \bar{P}_{1}^{-}e_0$ and $TV_+^\varphi=V_-^\varphi$, we must have $T\bar{P}_{2}^{+}e_0=\pm \bar{P}_{2}^{-}e_0$. Then, inserting $\bar{P}^\pm_2 e_0=\cos(\varphi)J_2 e_0 + \sin(\varphi) J_2 e^\pm_2$ in the last equality, and using $TJ_2=J_2T$, we deduce that $T e^+_2=e^-_2$. But this jointly with $Te_1^+=e_1^-$ yields a contradiction with the fact that $T$ is an orthogonal transformation of $\H^n$, because $\langle e^+_1,e^+_2\rangle\neq\langle e^-_1,e^-_2\rangle$ for all $\varphi\neq\pi/2$.
\end{proof}
\subsection{Subspaces of dimension four}\label{subsec:dim4}\hfill

The aim of this subsection is to classify $4$-dimensional real subspaces of $\H^n$ with constant quaternionic K\"ahler angle.

We start by restricting our attention to subspaces with $\varphi_1=0$.

\begin{proposition}
	\label{prop:angle0}
	Let $V\subset \mathbb{H}^n$ be a real subspace of dimension $4$ with constant quaternionic K\"ahler angle $(0,\varphi_2,\varphi_3)$.  Then, $\varphi_2=\varphi_3\in[0,\pi/2]$.
\end{proposition}
\begin{proof}
First of all, if  $\varphi_2=0$, then $\varphi_3=0$  by a combination of \cite[Proposition~9]{BB01} and the fact that subspaces with $\Phi(V)=(0,0,\pi/2)$ have dimension $3$ (see \S\ref{subsec:known_examples}).
Hence, let us assume that $\varphi_2\neq0$.
Lemma~\ref{lemma:basisk34} yields a basis $\{e_0,J_1 e_0,v_2,v_3\}$ of $V$, where $v_i=\cos(\varphi_i)J_ie_0+\sin(\varphi_i)J_ie_i$, for certain unit $e_i\in\H^n\ominus\H e_0$, $i\in\{2,3\}$. Therefore, a computation as in Equations (\ref{eq:omegaij3}), for each $i\in\{2,3\}$, gives
\begin{align*}
\Omega(v_i)_{11}&=\bigl(\cos(\varphi_2)\cos(\varphi_3)+\langle e_2,e_3\rangle\sin(\varphi_2)\sin(\varphi_3)\bigr)^2,
\\
\Omega(v_i)_{22}&=\cos(\varphi_i)^2+\langle J_3e_3,e_2\rangle^2\sin^2(\varphi_2)\sin^2(\varphi_3),
\\
\Omega(v_i)_{33}&=\cos(\varphi_i)^2+\langle J_2e_2,e_3\rangle^2\sin^2(\varphi_2)\sin^2(\varphi_3).
\end{align*}
Hence, by the isospectrality of $\Omega$,
	\[
	0=\tr(\Omega(v_2))-\tr(\Omega(v_3))=2\cos^2(\varphi_2)-2\cos^2(\varphi_3),
	\]
	from where we conclude $\varphi_2=\varphi_3$.
\end{proof}
In view of Proposition~\ref{prop:angle0}, all real subspaces $V$ of $\H^n$ with $\Phi(V)=(0,\varphi_2,\varphi_3)$ actually satisfy $\Phi(V)=(0,\varphi,\varphi)$. Note that such subspaces have been classified (see~\S \ref{subsec:known_examples}).

Thus, in the following results we will analyze the case $\varphi_1>0$.
We consider the basis of $V$ given in Lemma~\ref{lemma:basisk34}.

\begin{lemma}
	\label{complexhelp}
	Let $V\subset\H^n$ be a real subspace of dimension $4$ such that $\Phi(V)=(\varphi_1,\varphi_2,\varphi_3)$ with $\varphi_1>0$.  For each $i\in\{1,3\}$ with $\varphi_i\neq\pi/2$, we have  $\langle e_i, J_j e_j\rangle=0$ for all $j\in\{1,2,3\}$.
\end{lemma}
\begin{proof}
	According to Lemma~\ref{lemma:basisk34}, $e_0$ has K\"ahler angle $\varphi_i$ with respect to $V$ and $J_i\in \g{J}$ for each $i\in\{1,2,3\}$.
	Let us regard $\H^n$ as a complex Euclidean space $\C^{2n}$ whose  complex structure is $J_i$, for $i\in\{1,2,3\}$. By \cite[Theorem~2.7]{mathz} there is a non-empty finite subset $\Psi^i\subset[0,\pi/2]$ such that $V=\bigoplus_{\varphi\in \Psi^i} V^i_{\varphi}$ is a $\C$-orthonormal decomposition of $V$ and $V_\varphi^i\subset \C^{2n}$ is a real subspace with constant K\"ahler angle $\varphi\in\Psi^i$. It follows that any non-zero $v\in V^i_\varphi$ has K\"ahler angle $\varphi$ with respect to $V$ and $J_i$, and the minimum (resp.\ maximum) of $\Psi^i$ coincides with the minimum (resp.\ maximum) K\"ahler angle of a non-zero vector $v\in V$ with respect to $V$ and $J_i$.
	
	We claim that $\varphi_1\in \Psi^1$. On the one hand, if there existed $\varphi\in \Psi^1$ such that $\varphi<\varphi_1$, then there would be vectors in $V$ whose K\"ahler angle with respect to $V$ and $J_1\in\g{J}$ is $\varphi<\varphi_1$, thus contradicting the minimality of $\varphi_1$ by Lemma \ref{lemma:qKa}. On the other hand, if $\varphi>\varphi_1$ for all $\varphi\in\Psi^1$, then we would get a contradiction with the fact that $e_0$ has K\"ahler angle $\varphi_1$ with respect to $J_1$.  Analogously, we get that $\varphi_3\in \Psi^3$.

Now assume $\varphi_1\neq\pi/2$. By~\cite[p.~1190--1191]{mathz} and the discussion above, we have a decomposition $V=V_{\varphi_1}\oplus V_{\psi_1}$ into real subspaces of constant K\"{a}hler angle with respect to the complex structure $J_1$, where $\psi_1\in\Psi^1$ (the possibility $\psi_1=\varphi_1$ is allowed).
We also have that  $\bar{P}_{1}:=\pi_{V_{\varphi_1}}  J_1/\cos(\varphi_1)=\pi_V J_1/\cos(\varphi_1)\vert_{V_{\varphi_1}}$ defines a complex structure on $V_{\varphi_1}$.
As $e_0\in V_{\varphi_1}$, we get $V_{\varphi_1}=\spann_{\R}\{e_0,\bar{P}_1 e_0\}$.
Moreover, $\C V_{\varphi_1}\perp\C V_{\psi_1}$, so
$V_{\psi_1}=\spann_\R\{\bar{P}_2 e_0,\bar{P}_3 e_0\}$, and for $j\in\{2,3\}$, using Lemma~\ref{lemma:basisk34},
\[
0=\langle\bar{P}_1 e_0,J_1\bar{P}_j e_0\rangle
=\sin(\varphi_1)\sin(\varphi_j)\langle e_1,J_j e_j\rangle.
\]
Since $\varphi_1>0$, we get $\langle e_1,J_j e_j\rangle=0$.
A similar argument works for $\varphi_3$, if $\varphi_3\neq\pi/2$.
\end{proof}

Before addressing the classification, we  state a lemma that refines \cite[Lemma~5.1]{damekricci}.

\begin{lemma}\label{lemma:inequalities}
	Assume $0< \varphi_1\leq\varphi_2\leq\varphi_3\leq \pi/2$, and let $\varepsilon\in\{-1,1\}$. Then, there exists a subset $\{e_1,e_2,e_3\}$ of unit vectors of $\R^3$ with inner products
	\[
	\langle e_i, e_{i+1}\rangle= \frac{\varepsilon\cos(\varphi_{i+2})-\cos(\varphi_{i})\cos(\varphi_{i+1}) }{\sin(\varphi_{i})\sin(\varphi_{i+1})} \qquad \text{for each }i\in\{1,2,3\}
	\]
	if and only if $\cos(\varphi_1)+ \cos(\varphi_2) -\varepsilon\cos(\varphi_3)\leq 1$.
	
	Furthermore, the subspace $\spann_{\R}\{e_1,e_2,e_3\}$ has dimension $2$ if and only if $\cos(\varphi_1)+ \cos(\varphi_2) +\varepsilon\cos(\varphi_3)=1$, and dimension $3$ otherwise.
\end{lemma}

\begin{proof}
	A subset $\{e_1,e_2,e_3\}$ of the Euclidean space $\R^3$ satisfies the inner product relations in the statement if and only if the associated Gram matrix $G=(\langle e_i,e_j\rangle)_{1\leq i,j\leq 3}$ is positive semi-definite. This happens precisely when all principal minors of $G$ are non-negative; in this proof, by $G_{ij}$ we denote the matrix of order $2$ resulting from deleting the $i$-th row and the $j$-th column of $G$. Let $x_i:=\cos(\varphi_i)$ for each $i\in\{1,2,3\}$. Hence, $G$ is positive semi-definite if and only if $\det G_{ii}\geq 0$ for all $i\in\{1,2,3\}$ and $\det G\geq 0$.
We compute
\[
\det(G)=\frac{(\varepsilon+ x_1 - x_2 -x_3)(-\varepsilon+ x_1 + x_2 -x_3)(-\varepsilon+ x_1 - x_2 +x_3)(\varepsilon+ x_1 + x_2  + x_3) }{(1-x_1^2)(1-x_2^2)(1-x_3^2)}.
\]
Taking into account that $1> x_1\geq x_2\geq x_3\geq 0$, one can check that
$\det G\geq 0$ if and only if $-1+x_1+x_2-\varepsilon x_3\leq 0$.
Similarly, $\det(G_{ii})=(1-x_1^2-x_2^2-x_3^2+2\varepsilon x_1x_2x_3)/\prod_{j\neq i}(1-x_j^2)$, $i\in\{1,2,3\}$.
Hence, $\det(G_{ii})\geq 0$ if and only if
\begin{equation}\label{eq:ineq_xi}
1-x_1^2-x_2^2-x_3^2+2\varepsilon x_1x_2x_3\geq 0.
\end{equation}
Now, if $1> x_1\geq x_2\geq x_3\geq 0$, one can show that \eqref{eq:ineq_xi} holds provided that $-1+x_1+x_2-\varepsilon x_3\leq 0$. This completes the proof of the first claim of the lemma.

Assume that we are in the situation of the first assertion of the statement. Then $\{e_1,e_2,e_3\}$ spans a $3$-dimensional subspace if and only if $G$ is positive definite, which in this situation amounts to $\det G>0$. This happens precisely when $x_1+x_2-\varepsilon x_3<1$. Hence, the proof of the lemma will be complete if we show that $\spann_{\R}\{e_1,e_2,e_3\}$ cannot have dimension $1$.	Assume this is the case. Then the rank of $G$ is $1$. Hence $x_1=1-x_2+\varepsilon x_3$, and the minor $\det (G_{33})$ vanishes, i.e.
\[
0=\det (G_{33})=-\frac{2(1+\epsilon x_3)}{(1+x_2)(-2+x_2-\epsilon x_3)}.
\]
Therefore, $x_3=-\varepsilon$, which yields a contradiction. This finishes the proof.
\end{proof}

We are now in position to complete the description of $4$-dimensional real subspaces of $\H^n$ with constant quaternionic K\"ahler angle.

\begin{proposition}\label{prop:class_dim4}
	Let $V\subset \H^n$ be a real subspace of dimension $4$ and $e_0\in V$ a unit vector. Then $V$ has constant quaternionic K\"ahler angle $\Phi(V)=(\varphi_1,\varphi_2,\varphi_3)$, with $\varphi_1>0$, if and only if there is a canonical basis $\{J_1,J_2,J_3\}$ of $\g{J}$, $\varepsilon\in\{-1,1\}$, and unit vectors $e_1,e_2,e_3\in \H^n\ominus\H e_0$ with $e_i\in \H^n\ominus(\mathrm{Im}\,\H)e_j$, $i$, $j\in\{1,2,3\}$, such that
\begin{enumerate}[{\rm (i)}]
\item $0<\varphi_1\leq\varphi_2\leq\varphi_3\leq\pi/2$,
\item $\cos(\varphi_1)+ \cos(\varphi_2) -\varepsilon\cos(\varphi_3)\leq 1$,\label{prop:dim4:inequality}
\item for all $i\in\{1,2,3\}$ and indices modulo $3$,
\[\langle e_i, e_{i+1}\rangle= \frac{\varepsilon\cos(\varphi_{i+2})-\cos(\varphi_{i})\cos(\varphi_{i+1}) }{\sin(\varphi_{i})\sin(\varphi_{i+1})},
\]
\item $\{\cos(\varphi_i) J_i e_0 + \sin(\varphi_i) J_i e_i: i=0,1,2,3 \}$ is an orthonormal basis of $V$, where for simplicity we put $\varphi_0:=0$ and $J_0:=\Id$.
\end{enumerate}
	
	Moreover, if $V$ is as above, the K\"ahler angle of any non-zero $v\in V$ with respect to $J_i$ and $V$ is $\varphi_i$, for each $i\in\{1,2,3\}$.
\end{proposition}

\begin{proof}
In order to prove the necessity, let us assume that  $V$ is spanned by the basis described in Lemma \ref{lemma:basisk34} with $k=4$.  Notice that  $\varphi_2=\pi/2$ implies $\Phi(V)=(\varphi,\pi/2,\pi/2)$; such subspaces are classified (see \S \ref{subsec:known_examples}), and \cite[p.~232]{BB01}, together with some straightforward calculations, show that they can be spanned by a basis as above.  Thus, we can suppose $\varphi_1,\varphi_2\in (0,\pi/2)$.

Let us first assume $\varphi_3\neq \pi/2$. A long but elementary calculation, similar to the one used to obtain Equations~(\ref{eq:omegaij3}), using the isospectrality of $\Omega$, Remark~\ref{rem:ortho} and Lemma~\ref{complexhelp}, yields
\begin{align*}
\prod_{j=1}^3 \cos^2(\varphi_j)&= \det(\Omega(\bar{P}_i e_0))= \cos^2(\varphi_i)\prod_{\substack{j=1\\j\neq i}}^3 (\cos(\varphi_i)\cos(\varphi_{j}) + \langle e_i, e_{j}\rangle \sin(\varphi_i)\sin(\varphi_{j}))^2,
\end{align*}
for each $i\in\{1,2,3\}$. This implies
\begin{equation}\label{eq:relation_cosines}
\cos^2(\varphi_{i+2})=\left(\cos \left(\varphi _i\right) \cos(\varphi _{i+1}) + \left\langle e_i,e_{i+1}\right\rangle \sin \left(\varphi _i\right) \sin \left(\varphi _{i+1}\right) \right){}^2, \qquad i\in\{1,2,3\}.
\end{equation}
Using \eqref{eq:relation_cosines}, we can also calculate for $i\in\{1,3\}$
\[
\sum_{j=1}^3 \cos^2(\varphi_j)=\tr(\Omega(\bar{P}_i e_0))=\sum_{j=1}^3 \cos^2(\varphi_j) + \langle  e_2, J_i e_i\rangle^2 \sin^2(\varphi_i)\sin^2(\varphi_2),
\]	
which implies $\langle e_2, J_i e_i\rangle=0$, $i\in\{1,3\}$. This, along with Remark~\ref{rem:ortho} and Lemma~\ref{complexhelp}, shows that $e_i\in\H^n\ominus(\mathrm{Im}\,\H)e_j$, $i,j\in\{1,2,3\}$. Furthermore, \eqref{eq:relation_cosines} gives rise to the two possible expressions for $\langle e_i,e_{i+1}\rangle$ in the statement (corresponding to $\varepsilon=1$ or $\varepsilon=-1$). Note that such expressions are incompatible for a fixed $V$, that is, if for some $i\in\{1,2,3\}$ we have
\begin{align*}
\langle e_i, e_{i+1}\rangle
&{}= \frac{\cos(\varphi_{i+2})-\cos(\varphi_{i})\cos(\varphi_{i+1})}
{\sin(\varphi_{i})\sin(\varphi_{i+1})},  
&\langle e_{i+1}, e_{i+2}\rangle
&{}=-\frac{\cos(\varphi_{i})+\cos(\varphi_{i+1})\cos(\varphi_{i+2})}
{\sin(\varphi_{i+1})\sin(\varphi_{i+2})},
\end{align*}
then one can check that $\det\left(\Omega((e_0 + \bar{P}_{i+1} e_0)/\sqrt{2})\right)=0$, which gives a contradiction with the assumption $\varphi_3\neq \pi/2$. Finally, the inequality in item~(\ref{prop:dim4:inequality}) of the statement follows from Lemma~\ref{lemma:inequalities}.

Now assume that $\varphi_3=\pi/2$. Let $\{e_0, \bar{P}_1 e_0, \bar{P}_2 e_0, J_3 e_3\}$ be the orthonormal basis provided by Lemma~\ref{lemma:basisk34}. A similar computation as in  (\ref{eq:omegaij3}), using the isospectrality of $\Omega$ and Lemma~\ref{complexhelp}, 
yields
\begin{align*}
\sum_{i=1}^2 \cos^2(\varphi_i)={}
&{}\tr(\Omega(\bar{P}_1 e_0)) +\tr(\Omega(\bar{P}_2 e_0)) - \tr(\Omega(J_3 e_3))
\\
=&{}\cos^2(\varphi_1) + \cos^2(\varphi_2) + 2 \sin ^2(\varphi _1) \sin^2(\varphi _2) \langle J_1e_1,e_2\rangle^2
\\[1ex]
&{}+2\left(\cos(\varphi _1) \cos(\varphi _2)+\sin(\varphi _1) \sin(\varphi _2) \langle e_1,e_2\rangle\right)^2.
\end{align*}
Then,
\begin{equation}\label{eq:e2J1e2}
\langle e_2, J_1 e_1\rangle=0 \quad \text{ and }\quad \langle e_1, e_2 \rangle=- \cot(\varphi_1)\cot(\varphi_2).
\end{equation}
Also, using~\eqref{eq:e2J1e2}, if $i\in\{1,2\}$ we get
\[ 
0=\det\left(\Omega\left(\frac{1}{\sqrt{2}} e_0 + \frac{1}{\sqrt{2}} \bar{P}_i e_0\right)\right)
=\frac{1}{4} \langle e_3, J_i e_i\rangle^2 \cos^2(\varphi_1) \cos^2(\varphi_2)\sin^2(\varphi_i).   \]
Thus,
\begin{equation}\label{eq:e3Jiei}
\langle e_3, J_i e_i\rangle=0, \quad i\in\{1,2\}.
\end{equation}
Taking into account~\eqref{eq:e2J1e2} and~\eqref{eq:e3Jiei}, we can calculate
\begin{align*}\sum_{i=1}^2 \cos^2(\varphi_i)&=\tr(\Omega(\bar{P}_1 e_0))=\cos^2(\varphi_1) + \langle e_1, e_3\rangle^2\sin^2(\varphi_1),
\\
\sum_{i=1}^2 \cos^2(\varphi_i)&=\tr(\Omega(\bar{P}_2 e_0))=\cos^2(\varphi_2) + \sin^2(\varphi_2)(\langle e_2, e_3\rangle^2 +\langle e_2, J_3 e_3\rangle^2),
\end{align*}
whence
\begin{equation}\label{eq:e1e3}\langle e_1, e_3\rangle=\varepsilon\cos(\varphi_2)/\sin(\varphi_1) \quad  \text{and}\quad \cos^2(\varphi_1)=\sin^2(\varphi_2)(\langle e_2, e_3\rangle^2 +\langle e_2, J_3 e_3\rangle^2),
\end{equation}
for some $\varepsilon\in\{-1,1\}$.
Using these relations we compute
\[\sum_{i=1}^2 \cos^2(\varphi_i)=\tr \left(\Omega\left(\frac{1}{\sqrt{2}} \bar{P}_1 e_0 + \frac{1}{\sqrt{2}} \bar{P}_2 e_0\right)\right)=\sum_{i=1}^2 \cos^2(\varphi_i) + \varepsilon \frac{1}{2}  \langle  e_2, J_3 e_3\rangle \sin(2\varphi_2),\]
from where (note that we are assuming $\varphi_2\neq\pi/2$)
\begin{equation}\label{eq:e2J3e3}
\langle e_2, J_3 e_3\rangle=0 \quad\text{and}\quad  \langle e_2, e_3\rangle=\varepsilon'\cos(\varphi_1)/\sin(\varphi_2),
\end{equation}
for some $\varepsilon'\in\{-1,1\}$. Remark~\ref{rem:ortho}, Lemma~\ref{complexhelp} and Equations~\eqref{eq:e2J1e2}, \eqref{eq:e3Jiei}, \eqref{eq:e1e3} and~\eqref{eq:e2J3e3} imply $e_i\in\H^n\ominus(\mathrm{Im}\,\H)e_j$, $i,j\in\{1,2,3\}$. Furthermore, if $\varepsilon'=-\varepsilon$,
we have that $0$ is an eigenvalue of $\Omega\left((e_0 +J_3 e_3)/\sqrt{2}\right)$ with double multiplicity, yielding a contradiction with the fact $\varphi_2\neq \pi/2$. Hence, $\varepsilon'=\varepsilon$ which, along with Lemma~\ref{lemma:inequalities}, concludes the proof of the necessity in the statement.

The converse implication follows from verifying by direct calculation that the matrix $\Omega(v)$ is diagonal with diagonal entries $\cos^2(\varphi_1)$, $\cos^2(\varphi_2)$, $\cos^2(\varphi_3)$, for any unit $v$ spanned by the basis of $V$ given in the statement. This also proves the final claim of the proposition.
\end{proof}

\begin{remark}\label{rem:V+-}
In view of Proposition~\ref{prop:class_dim4}, there can be zero, one or two  types of $4$-dimensional real subspaces $V$ of $\H^n$ with $\Phi(V)=(\varphi_1,\varphi_2,\varphi_3)$, $\varphi_1>0$, depending on whether the triple $(\varphi_1,\varphi_2,\varphi_3)$ satisfies $\cos(\varphi_1)+ \cos(\varphi_2) -\cos(\varphi_3)> 1$, $\cos(\varphi_1)+ \cos(\varphi_2) -\cos(\varphi_3)\leq 1< \cos(\varphi_1)+ \cos(\varphi_2) +\cos(\varphi_3)$, or $\cos(\varphi_1)+ \cos(\varphi_2) +\cos(\varphi_3)\leq 1$, respectively. Thus, it will be convenient to denote by $V_+$ and $V_-$ the subspaces described in Proposition~\ref{prop:class_dim4} with $\varepsilon=1$ or $\varepsilon=-1$, respectively. Note that such subspaces depend on the triple $(\varphi_1,\varphi_2,\varphi_3)$, but we do not specify this in the notation for the sake of simplicity.

Observe that, if $\varphi_3=\pi/2$, $V_+$ and $V_-$ are actually equivalent, i.e.\ there exists $T\in \Sp(1)\Sp(n)$ such that $TV_+=V_-$. Indeed, one can take a $T\in \Sp(1)\Sp(n)$ that commutes with $J_i$ for all $i\in\{1,2,3\}$, fixes each $e_i$ with $i\in\{0,1,2\}$, and sends $e_3$ to $-e_3$. For convenience, from now on we will say that any $4$-dimensional real subspace of $\H^n$ with constant quaternionic K\"ahler angle $(\varphi_1,\varphi_2,\pi/2)$ is of type $V_+$, and not of type $V_-$.

In order to encompass all examples of $4$-dimensional subspaces with constant quaternionic K\"ahler angle into the $V_\pm$-notation, we have to consider the case $\varphi_1=0$ analyzed in Proposition~\ref{prop:angle0}. Thus, we adopt the convention that any $4$-dimensional real subspace $V$ with $\Phi(V)=(0,\varphi,\varphi)$, $\varphi\in[0,\pi/2]$, is of type $V_+$, and not of type $V_-$.
\end{remark}

\begin{remark}\label{rem:P1P2}
The choice of the $\pm$-notation in Remark~\ref{rem:V+-} is motivated by certain important property of these subspaces that we now explain. Assume $\varphi_3\neq \pi/2$. The last claim of Proposition~\ref{prop:class_dim4} enables us to reproduce the discussion in~\S\ref{subsec:factorization} applied to $V=V_\pm$, and hence, $\bar{P}_i=\pi_{V_\pm}J_i/\cos(\varphi_i)$, $i\in\{1,2,3\}$, determine a $\mathsf{Cl}(3)$-module structure on $V_\pm$, which must be irreducible since $\dim_\R V_\pm=4$. By the classification of Clifford modules, either $\bar{P}_1\bar{P}_2=\bar{P}_3$ (and hence $\bar{P}_i\bar{P}_{i+1}=\bar{P}_{i+2}$ for all $i\in\{1,2,3\}$) or $\bar{P}_1\bar{P}_2=-\bar{P}_3$  (and hence $\bar{P}_i\bar{P}_{i+1}=-\bar{P}_{i+2}$ for all $i\in\{1,2,3\}$). One can easily check using the basis of $V_\pm$ in Proposition~\ref{prop:class_dim4} that $V_+$ satisfies precisely the former relation, whereas $V_-$ satisfies the latter.
\end{remark}

\begin{remark}\label{rem:smaller_n_k4}
	Let $V$ be a real subspace of dimension $4$ in $\H^n$, $n\ge 4$,  with $\Phi(V)=(\varphi_1,\varphi_2,\varphi_3)$. If $\varphi_1=0$,  by Proposition \ref{prop:angle0} we have $\Phi(V)=(0,\varphi,\varphi)$ for some $\varphi\in[0,\pi/2]$. In this case, when $\varphi=0$, $V$ is quaternionic, i.e.~ $V=\H v$, for some non-zero vector $v\in V$, whereas if $\varphi> 0$, $V$ cannot fit inside a quaternionic line $\H$, but can be placed in some $\H^2$ (see \S \ref{subsec:known_examples}), and thus $\H V= \H^2$.
	
	Now assume $\varphi_1 >0$. By Proposition \ref{prop:class_dim4} and Lemma~\ref{lemma:inequalities}, $V$ can be placed in some $\H^3$ if and only if $V=V_{+}$ and $\cos(\varphi_1) + \cos(\varphi_2) - \cos(\varphi_3)=1$, or if $V=V_{-}$ and $\cos(\varphi_1) + \cos(\varphi_2) + \cos(\varphi_3)=1$; in this case $\H V=\H^3$. Otherwise we have $\H V=\H^4$.
\end{remark}

We end this section by showing that $V_+$ and $V_-$ are not equivalent.
\begin{proposition}\label{prop:ineqk=4}
	There does not exist $T\in \Sp(1)\Sp(n)$ such that $TV_+=V_-$.
\end{proposition}
\begin{proof}
We can assume that $\Phi(V_+)=\Phi(V_-)$ since the quaternionic K\"ahler angle is preserved by transformations in $\Sp(1)\Sp(n)$. We also assume $\varphi_3\neq\pi/2$ in view of Remark~\ref{rem:V+-}. We consider the bases for $V_\pm$ given in Proposition~\ref{prop:class_dim4}, where we use the notation $e_i^\pm$ accordingly, and assume without restriction of generality that the canonical basis $\{J_1,J_2,J_3\}$ of $\g{J}$ used is the same in both cases.

Let us suppose that there is $T\in \Sp(1)\Sp(n)$ such that $TV_+=V_-$. Denote by $\pi_+$ and $\pi_-$ the orthogonal projections onto $V_+$ and $V_-$, respectively. By assumption $T\pi_+=\pi_-T$. Let $\{J_1',J_2',J_3'\}$ be the canonical basis of $\g{J}$ given by $J_i'=TJ_iT^{-1}$, $i\in\{1,2,3\}$, and denote $P_i^+=\pi_+J_i$ and $P_i'=\pi_- J_i'$, $i\in\{1,2,3\}$. Then, for any unit vector $w \in V_-$ and $i,j\in\{1,2,3\}$, we have
	\begin{align*}
	\langle P_i'w,P_j'w\rangle&=\langle \pi_-TJ_iT^{-1}w,\pi_-TJ_jT^{-1}w\rangle = \langle T\pi_+J_iT^{-1}w,T\pi_+J_jT^{-1}w\rangle \\
	&= \langle TP_i^+T^{-1}w,TP_j^+T^{-1}w\rangle= \langle P_i^+T^{-1}w, P_j^+ T^{-1}w\rangle=\cos^2(\varphi_i)\delta_{ij},
	\end{align*}
	where in the last equality we have used the last claim of Proposition~\ref{prop:class_dim4} applied to $V_+$.
	
	Thus, the canonical basis $\{J_1',J_2',J_3'\}$ of $\g{J}$ diagonalizes the bilinear form $L^-_w$ (given in Proposition~\ref{lemma:qKa}) associated with the subspace $V_-$, for any unit vector $w\in V_-$.
	By the last claim of Proposition~\ref{prop:class_dim4} applied to $V_-$, the basis $\{J_1,J_2,J_3\}$ also has this property. Hence, there exists an orthogonal matrix $A\in \SO(3)$ such that
	\[
	(J_1',J_2',J_3')=(J_1,J_2,J_3)A
	\]
	and $A$ commutes with the diagonal matrix with diagonal entries $(\varphi_1,\varphi_2,\varphi_3)$.
	Then $V_-$ coincides with the span of
	\[\{e_0^-, \cos(\varphi_1) J_1' e_0^- + \sin(\varphi_1) J_1' e_1^-,\cos(\varphi_2) J_2' e_0^- + \sin(\varphi_2) J_2' e_2^-, \cos(\varphi_3) J_3' e_0^- + \sin(\varphi_3) J_3' e_3^- \},\]
	where in this basis we have just changed $J_i$ by $J_i'$ in the original basis of $V_-$. Since for $V_-$ we had $\bar{P}_1^-\bar{P}_2^-=-\bar{P}_3^-$ by Remark~\ref{rem:P1P2}, where $\bar{P}_i^-=\pi_-J_i/\cos(\varphi_i)$, $i\in\{1,2,3\}$, we also have $\bar{P}'_1\bar{P}'_2=-\bar{P}'_3$, where $\bar{P}_i'=P_i'/\cos(\varphi_i)=\pi_-J_i'/\cos(\varphi_i)$, $i\in\{1,2,3\}$.
	
	However, denoting $\bar{P}_i^+=\pi_+J_i/\cos(\varphi_i)$, $i\in\{1,2,3\}$, which again by Remark~\ref{rem:P1P2} satisfy $\bar{P}_1^+\bar{P}_2^+=\bar{P}_3^+$, we obtain:
\begin{align*}
\bar{P}_1'\bar{P}_2'&=\frac{1}{\cos(\varphi_1)\cos(\varphi_2)}\pi_-J_1'\pi_-J_2'
=\frac{1}{\cos(\varphi_1)\cos(\varphi_2)}\pi_-TJ_1T^{-1}\pi_-TJ_2T^{-1}
\\
&=\frac{1}{\cos(\varphi_1)\cos(\varphi_2)}T\pi_+J_1\pi_+J_2T^{-1}=T\bar{P}_1^+\bar{P}_2^+T^{-1}
=T\bar{P}_3^+T^{-1}
\\
&=\frac{1}{\cos(\varphi_3)}T\pi_+J_3T^{-1} = \frac{1}{\cos(\varphi_3)}\pi_-TJ_3T^{-1}=
\frac{1}{\cos(\varphi_3)}\pi_-J_3'=\bar{P}_3',
\end{align*}
	which leads to a contradiction with  $\bar{P}'_1\bar{P}'_2=-\bar{P}'_3$.
\end{proof}

\section{Inhomogeneous isoparametric hypersurfaces\\ with constant principal curvatures}\label{inhomo}

In this section we will investigate when an $\H$-orthogonal sum of copies of the subspaces $V_\pm$ introduced in the previous section gives rise to a protohomogeneous real subspace of~$\H^n$. In particular, we will obtain subspaces with constant quaternionic K\"ahler angle that are not protohomogeneous. As a consequence of \cite[Theorem 4.5]{damekricci}  these subspaces give rise to examples of inhomogeneous isoparametric hypersurfaces with constant principal curvatures in quaternionic hyperbolic spaces.

Let us consider a real subspace $V$ of $\H^n$ such that
\begin{enumerate}[{\rm (C1)}]
	\item $V=\bigoplus_{r=1}^{l} V_r$, where
	\item $\dim_\R V_r=4$, for each $r\in\{1,\dots,l\}$,
	\item $V_r$ and $V_s$ are $\H$-orthogonal for every $r,s\in\{1,\dots,l\}$, $r\neq s$,
	\item $\Phi(V_r)=(\varphi_1,\varphi_2,\varphi_3)$, for all $r\in\{1,\dots,l\}$,
	\item $\{J_1,J_2,J_3\}$ is a canonical basis of $\g{J}$ such that every non-zero vector in $V_r$, $r\in\{1,\dots, l\}$, has K\"ahler angle $\varphi_i$ with respect to $J_i$ for each $i\in\{1,2,3\}$.
\end{enumerate}
Then, Lemma~\ref{lemma:factorization} guarantees that $\Phi(V)=(\varphi_1,\varphi_2,\varphi_3)$.
By Proposition~\ref{prop:class_dim4}, Remark~\ref{rem:V+-} and Proposition~\ref{prop:ineqk=4}, each factor $V_r$ is either equivalent  to $V_{+}$ or to $V_{-}$. Then, if we define $l_+$ and $l_-$ as the number of subspaces in the decomposition of $V$ equivalent to $V_+$ and to $V_-$, respectively, we have $l=l_++l_-$. In this situation we will say that the real subspace $V$ has \emph{type $(l_+,l_-)$}.

We claim that the type of $V$ is well defined for real subspaces of $\H^n$ in the conditions (C1-5) above. If $\varphi_3=\pi/2$, then by Remark~\ref{rem:V+-} the type of $V$ is $(l,0)$. Let us assume that $\varphi_3\neq\pi/2$. As usual, we let
$\bar{P}_i=\pi_V J_i/\cos(\varphi_i)$,  for each $i\in\{1,2,3\}$; since $\H V_r\perp \H V_s$ for $r\neq s$, we have $\bar{P}_i\vert_{V_r}=\pi_{V_r}J_i/\cos(\varphi_i)$. Thus, it follows from Remark~\ref{rem:P1P2} that $\bar{P}_1 \bar{P}_{2}\vert_{V_r}=\bar{P}_{3}\vert_{V_{r}}$ or  $\bar{P}_1 \bar{P}_{2}\vert_{V_r}=-\bar{P}_{3}\vert_{V_{r}}$, depending on whether $V_r$ is equivalent to $V_+$ or $V_-$, respectively. Hence,  $\dim_\R \Ker(\bar{P}_1 \bar{P}_{2} - \bar{P}_{3})=4l_{+}$ and  $\dim_\R \Ker(\bar{P}_1 \bar{P}_{2} + \bar{P}_{3})=4l_{-}$. Moreover, the type is independent of the canonical basis of $\g{J}$ chosen. Indeed, if $\{J_1',J_2',J_3'\}$ is another canonical basis satisfying (C5), then there exists an orthogonal matrix $A\in\SO(3)$ such that $(J_1',J_2',J_3')=(J_1,J_2,J_3)A$ and commuting with the diagonal matrix with diagonal entries $(\varphi_1,\varphi_2,\varphi_3)$, and one can easily argue (similarly as in the proof of Proposition~\ref{prop:ineqk=4}) that $\Ker(\bar{P}_1' \bar{P}_{2}' \pm \bar{P}_{3}')=\Ker(\bar{P}_1 \bar{P}_{2} \pm \bar{P}_{3})$. All in all, the type is well defined. More than that, a slight modification of the previous argument shows that two real subspaces $V$ and $W$ of $\H^n$ in the conditions (C1-5) are equivalent if and only if $\Phi(V)=\Phi(W)$ and their types coincide.

\begin{proposition}\label{prop:protosum}
	Let $V$ be a real subspace of $\H^n$ satisfying conditions \emph{(C1-5)}. Then, $V$ is protohomogeneous if and only if the type of $V$ is $(l,0)$ or $(0,l)$.
\end{proposition}
\begin{proof}
	Assume that $V$ is protohomogenous. By conditions (C1-5), $V=\bigoplus_{r=1}^l V_r$, where each factor $V_r$ is equivalent either to $V_+$ or to $V_-$. 
	If $\varphi_3=\pi/2$, by Remark~\ref{rem:V+-} each factor $V_r$ is equivalent to $V_+$, whence $V$ has type $(l,0)$.
	
	Let us suppose that $V$ has type $(l^+,l^-)$ where $l^+,l^-\ge 1$. In this case, $k=4l\geq 8$. Let $r,s\in\{1,\dots, l\}$, $r\neq s$, be such that $V_r$ is equivalent to $V_+$, and $V_s$ is equivalent to $V_-$. Let $v_+$ and $v_-$ be unit vectors in $V_r$ and $V_s$, respectively.
	Since $V$ is protohomogeneous, there is $T\in \Sp(1)\Sp(n)$ such that $TV=V$ and $T v_+=v_-$. Now, since $k\ge 8$, by Proposition~\ref{prop:H_subgroup_spn} we can assume that $T$ is such that $TJ_i=J_i T$ for each $i\in\{1,2,3\}$. Then,
	\begin{align*}
	T P_i v_+=T \pi_V J_i v_+=\pi_V T J_i v_+=\pi_V  J_i T v_+=\pi_V  J_i v_-=P_iv_-,
	\end{align*}
	for each $i\in\{1,2,3\}$. Then, $T$  sends the subspace $V_r=\spann\{v_+,P_1v_+,P_2v_+,P_3v_+\}$ onto $V_s=\spann\{v_-,P_1v_-,P_2v_-,P_3v_-\}$. This yields a contradiction with Proposition \ref{prop:ineqk=4}.

	Now we will prove the converse.	Let $V$ be of type $(l,0)$ or $(0,l)$.
	We can assume $\varphi_1>0$. Otherwise, by Proposition~\ref{prop:angle0} we have $\Phi(V)=(0,\varphi,\varphi)$ with $\varphi\in[0,\pi/2]$, and then $V$ is protohomogeneous (see \S \ref{subsec:known_examples}). We can also assume   $\varphi_2<\pi/2$. Otherwise, $\Phi(V)=(\varphi,\pi/2,\pi/2)$ for $\varphi\in[0,\pi/2]$, and then $V$ would again be protohomogeneous (see~\S \ref{subsec:known_examples}).
	
	As usual, consider the transformations $\bar{P}_i=\pi_V  J_i/\cos(\varphi_i)$ for each $i\in\{1,2,3\}$ with $\varphi_i\neq \pi/2$, and define $\bar{P}_3:=\bar{P}_1  \bar{P}_2$  if $\varphi_3=\pi/2$. By the characterization of type, we have $\bar{P}_1  \bar{P}_2=\varepsilon \bar{P}_3$, where $\varepsilon=1$ if the type of $V$ is $(l,0)$, and $\varepsilon=-1$ if the type of $V$ is $(0,l)$. Thus, taking into account condition (C5) and the discussion in~\S\ref{subsec:factorization} (or alternatively by the very definition of $V$), we deduce that $\{\bar{P}_1,\bar{P}_2,\varepsilon\bar{P_3}\}$ is a canonical basis of a quaternionic structure on~$V$.	

	Let $v_1,w_1\in V$ be arbitrary unit vectors. Then, $\{ \bar{P}_i v_1\}_{i=0}^3$ and  $\{ \bar{P}_i w_1\}_{i=0}^3$, where $\bar{P}_0=\Id$, are $\R$-orthonormal bases for some $4$-dimensional subspaces $V_{v_1}$ and $V_{w_1}$ of $V$, respectively. By construction, $V_{v_1}$ and $V_{w_1}$ are $\bar{P}_i$-invariant for each $i\in\{1,2,3\}$, and then
	$\{\bar{P}_1,\bar{P}_2,\varepsilon\bar{P_3}\}$ is a canonical basis of a quaternionic structure when restricted to $V_{v_1}$ and to $V_{w_1}$. Moreover, every non-zero vector in $V_{v_1}$ or $V_{w_1}$ has K\"ahler angle $\varphi_i$ with respect to $J_i$, for each $i\in \{1,2,3\}$, by Proposition~\ref{prop:class_dim4}. In conclusion, $V_{v_1}$ and $V_{w_1}$ are both equivalent either to $V_{+}$ or to $V_{-}$, depending on whether the type of $V$ is $(l,0)$ or $(0,l)$, respectively.
	
Proceeding inductively we can choose unit vectors $v_2,w_2,\ldots, v_l, w_l$ and define decompositions $V=\bigoplus_{r=1}^l V_{v_r}$ and $V=\bigoplus_{r=1}^lV_{w_r}$ satisfying (C1-5), and such that $V_{v_r}=\spann\{\bar{P}_i v_r\}_{i=0}^3$ and $V_{w_r}=\spann\{\bar{P}_i w_r\}_{i=0}^3$ for each $r\in\{1,\ldots, l\}$. Furthermore, all these subspaces $V_{v_r}$ and $V_{w_r}$ are equivalent to either $V_{+}$ or to $V_{-}$, depending on the type of $V$. Thus, for each $r\in\{1,\ldots, l\}$ there exist $T_r\in \Sp(n)\subset \Sp(1)\Sp(n)$ such that $T_rV_{v_r}=V_{w_r}$, $T_r(\H V_{v_r})=\H V_{w_r}$, $T_r\vert_{\H^n\ominus\H V_{v_r}}=\Id$, and $T_rJ_i=J_iT_r$ for each $i\in\{1,2,3\}$. Now let $e_{0}=v_1$, $e_{1}$, $e_{2}$, $e_{3}$ be the unit vectors given in Proposition~\ref{prop:class_dim4} for the subspace $V_{v_1}$, and similarly $f_{0}=w_1$, $f_{1}$, $f_{2}$, $f_{3}$ the unit vectors associated with the subspace $V_{w_1}$. Since $\langle e_i, e_j\rangle=\langle f_i,f_j\rangle$ for all $i,j\in\{0,1,2,3\}$, and both sets of vectors span a totally real subspace of $\H^n$, there exists $T_1'\in \Sp(n)\subset\Sp(1)\Sp(n)$ satisfying, in addition to the properties of the previously constructed $T_1$, the relations $T_1'e_j=f_j$ for each $j\in\{0,1,2,3\}$ (in particular $T_1'v_1=w_1$), and $T_1'J_i=J_iT_1'$ for each $i\in\{1,2,3\}$. Therefore, the composition $T=T_1' T_2\dots T_r\in \Sp(n)$ satisfies $TV=V$ and $Tv_1=w_1$, which shows that $V$ is protohomogeneous.
\end{proof}

\begin{remark}
	We observe that an $\H$-orthogonal direct sum of real subspaces of dimension~4 with the same constant quaternionic K\"{a}hler angle (i.e.\ any subspace $V$ satisfying (C1-4)) is protohomogeneous if and only if any two factors are congruent under an element of $\Sp(n)$. The direct implication follows from a combination of the simultaneous diagonalization result in Corollary~\ref{prop:in_spn} (which implies condition (C5)), Lemma~\ref{lemma:factorization} (which guarantees that $V$ has constant quaternionic K\"ahler angle) and Proposition~\ref{prop:protosum} (whose proof implies that any two factors are congruent under an element of $\Sp(n)$). The converse follows from a direct calculation using the fact that condition (C5) is satisfied if any two factors are congruent under an element of $\Sp(n)$, as the elements of $\Sp(n)$ commute with any $J\in\g{J}$.
\end{remark}

\begin{remark}
	Propositions~\ref{prop:class_dim4} and~\ref{prop:protosum} imply in particular that any $4$-dimensional real subspace of $\H^n$ with constant quaternionic K\"ahler angle is protohomogeneous. Recall that, by Propositions~\ref{prop:sk=3} and~\ref{proto3}, the same happens with any $3$-dimensional real subspace of $\H^n$ with constant quaternionic K\"ahler angle. This, along with Proposition~\ref{prop:hairyball} and the well-known protohomogeneity of subspaces $V$ with $\Phi(V)=(\varphi,\pi/2,\pi/2)$, $\varphi\in[0,\pi/2]$, implies that any real subspace of $\H^n$ with constant quaternionic K\"ahler angle and dimension $k\neq 4l$, for all $l\in\mathbb{N}$, $l\geq 2$, is protohomogeneous.
\end{remark}

An immediate consequence of Proposition~\ref{prop:protosum}, along with Lemma~\ref{lemma:factorization}, is the existence of non-protohomogeneous subspaces with constant quaternionic K\"ahler angle.

\begin{corollary}\label{cor:nonproto}
	A real subspace $V$ of $\H^n$, satisfying conditions \emph{(C1-5)} above in this section and of type $(l_+,l_-)$ with $l_+,l_-\geq 1$, has constant quaternionic K\"ahler angle but is not protohomogeneous.
\end{corollary}

Apart from the purely linear algebraic relevance of the examples described in Corollary~\ref{cor:nonproto}, our interest in them stems from the theory of isoparametric hypersurfaces in symmetric spaces of non-compact type, which we briefly describe now in the particular case of the quaternionic hyperbolic space $\H H^{n+1}$; we refer to~\cite{damekricci,DiDoSa19} for more details.

Following the notation in \S\ref{subsec:symmetric_spaces}, let $M=\H H^{n+1}=G/K$, where $G=\Sp(1,n+1)$, and $K=\Sp(1)\times\Sp(n+1)$ is the isotropy group at some base point $o\in \H H^{n+1}$. Let $AN$ be the solvable part of the Iwasawa decomposition of $G=\Sp(1,n)$, and $\g{a}\oplus\g{n}=\g{a}\oplus\g{g}_\alpha\oplus\g{g}_{2\alpha}$ its Lie algebra, where $\g{g}_\alpha\cong\mathbb{H}^n$.

Given any non-zero real subspace $V$ of $\g{g}_\alpha\cong\H^n$, we define $S_{V}$ as the connected subgroup of $AN$ with Lie algebra
\[
\mathfrak{s}_V=\mathfrak{a} \oplus (\g{g}_\alpha\ominus V)\oplus \mathfrak{g}_{2\alpha}.
\]
Then, by \cite[Theorem 4.5]{damekricci}, the orbit of $S_V$ through the base point $o$, together with the distance tubes around it, constitute an isoparametric family of hypersurfaces on~$\H H^{n+1}$:
\begin{theorem}\label{th:particularcasedamekricci}
	The tubes of any radius around the submanifold $S_{V}\cdot o$ are isoparametric hypersurfaces of $\H H^{n+1}$. Moreover, they have constant principal curvatures if and only if $V$ has constant quaternionic K\"{a}hler angle in $\H^n$.
\end{theorem}

As a consequence we get Theorem~C.

\begin{proof}[Proof of Theorem~C]
The combination of Theorem~\ref{th:bt:classification}, Corollary~\ref{cor:nonproto} and Theorem~\ref{th:particularcasedamekricci} guarantees the existence of inhomogeneous isoparametric hypersurfaces with constant principal curvatures in quaternionic hyperbolic spaces.

We note that this construction does not provide any such example in $\H H^{n+1}$ with $n\leq 6$, but it does so for $n\geq 7$.
This follows from Corollary~\ref{cor:nonproto} along with the fact that, by Remark~\ref{rem:smaller_n_k4}, the lowest integer $n$ such that $\H^n$ admits a real subspace $V$ of type $(l_+,l_-)$ with $l_+,l_-\geq 1$ is $n=7$. Indeed, we can take $V=V_+\oplus V_-\subset\H^n$ satisfying conditions (C1-5), for any triple $(\varphi_1,\varphi_2,\varphi_3)$, $\varphi_3\neq\pi/2$, such that $\cos(\varphi_1)+\cos(\varphi_2)+\cos(\varphi_3)=1$ if $n=7$, or such that $\cos(\varphi_1)+\cos(\varphi_2)+\cos(\varphi_3)\leq 1$ if $n\geq 8$.
\end{proof}

\section{Proofs of Theorems~A and~B}\label{sec:proof}

In this section we conclude the proof of the classification of protohomogeneous subspaces of any dimension $k>0$ in $\H^n$ by providing their moduli space.

\begin{proof}[Proof of Theorem~A]
We recall from the statement of Theorem~A the definition of the sets  $\Lambda=\{(\varphi_1,\varphi_2, \varphi_3)\in [0,\pi/2]^3: \varphi_1\leq \varphi_2\leq \varphi_3 \}$, and
\allowdisplaybreaks
\begin{align*}
\mathfrak{R}_3^+&{}=\{ (\varphi,\varphi,\pi/2)\in \Lambda :\varphi\in[0,\pi/2]  \},\\
\mathfrak{R}_3^-&{}=\{ (\varphi,\varphi,\pi/2)\in \Lambda :\varphi\in[\pi/3,\pi/2)   \},\\
\mathfrak{R}^+_4&{}=\{  (\varphi_1,\varphi_2,\varphi_3)\in \Lambda :\cos(\varphi_1)+\cos(\varphi_2)-\cos(\varphi_3)\leq 1 \},
\\
\mathfrak{R}^-_4&{}=\{  (\varphi_1,\varphi_2,\varphi_3)\in \Lambda :\cos(\varphi_1)+\cos(\varphi_2)+\cos(\varphi_3)\leq 1,\,\varphi_3\neq \pi/2 \},
\\
\mathfrak{S}&{}=\{(\varphi_1,\varphi_2,\varphi_3)\in \Lambda : \cos(\varphi_1)+\cos(\varphi_2)+\varepsilon\cos(\varphi_3)=1,\text{ for }\varepsilon=1\text{ or }\varepsilon=-1 \}.
\end{align*}

Note that $\g{R}^\pm_3$ (resp.\ $\g{R}^\pm_4$) is the set of possible triples that arise as quaternionic K\"ahler angles of the $3$-dimensional (resp.\ $4$-dimensional) subspaces $V_\pm^\varphi$ (resp.\ $V_\pm$) introduced in Remark~\ref{rem:k3smaller_n} (resp.\ Remark~\ref{rem:V+-}). Notice that $\mathfrak{R}^-_{3} \subset  \mathfrak{R}^+_{3}$, $\mathfrak{R}^-_{4} \subset  \mathfrak{R}^+_{4}$, $\g{S}\subset \g{R}^+_4\cup\g{R}^-_4$, and $\g{R}^-_4$ is precisely the set of triples for which there exist non-protohomogeneous subspaces with constant quaternionic K\"ahler angle.

Let $V$ be a non-zero protohomogeneous subspace of real dimension $k$ in $\H^n$.
The proof of Theorem~A follows from the discussion of the following four cases:
\begin{enumerate}[{\rm (1)}]
	\item \emph{Case $k\equiv 0 \,\mathrm{(mod\; 4)}$}. By Corollary~\ref{prop:in_spn} and Lemma~\ref{lemma:factorization}, $V$ satisfies conditions {(C1-5)} in Section~\ref{inhomo}, and by Proposition~\ref{prop:protosum}, $V$ is of type  $(k/4,0)$ or $(0,k/4)$. 
	Put $V=\bigoplus_{r=1}^{k/4} V_r$ as in (C1). Now, by Remark~\ref{rem:smaller_n_k4} we have  $\dim_{\H}(\H V_r)\in\{1,2,3,4\}$, depending on the value of the triple $\Phi(V)$. Thus, combining this with the fact that $\H V_r \perp \H V_s=0$ for  $r\neq s$, we can distiguish four subcases of relative sizes of $n$ and $k$, and determine the possible triples $\Phi(V)$ for each subcase:
	\begin{enumerate}[{\rm ({1}a)}]
		
		\item If $k>2n$, then $\Phi(V)=(0,0,0)$.
		
		\item If $4n/3< k\leq 2n$, then $\Phi(V)\in\{(0,\varphi,\varphi)\in \Lambda : \varphi\in[0,\pi/2]\}$.
		
		\item If $n< k\leq 4n/3$, then $\Phi(V)\in \g{S}$.
		
		\item Let us assume that $k\leq n$. If $V$ is of type $(k/4,0)$, then $\Phi(V)\in \mathfrak{R}^+_{4}$, whereas if $V$ is of type $(0,k/4)$, then $\Phi(V)\in \mathfrak{R}^-_{4}$. Observe that for each triple $(\varphi_1,\varphi_2,\varphi_3)$ in $\mathfrak{R}^-_{4}$ (resp.\ in $\g{R}^+_4\setminus\g{R}^-_4$) we have exactly two (resp.\ one) inequivalent protohomogeneous subspaces $V$ of dimension $k$ with $\Phi(V)=(\varphi_1,\varphi_2,\varphi_3)$.
	\end{enumerate}
	\item \emph{Case $k$ odd, $k\neq 3$}. By Proposition \ref{prop:hairyball} we have $\Phi(V)=(\pi/2,\pi/2,\pi/2)$. Hence, by the classification of totally real subspaces, we must have $k\leq n$.
	
	\item \emph{Case $k\equiv 2 \,\mathrm{(mod\; 4)}$}. By Proposition \ref{prop:hairyball} we have $\Phi(V)=(\varphi,\pi/2,\pi/2)$, for some $\varphi\in[0,\pi/2]$. These examples are classified (see \S \ref{subsec:known_examples}). Thus, we must have $k\leq 2n$. Furthermore, when $n< k \leq 2n$, we have $\Phi(V)=(0,\pi/2,\pi/2)$, whereas when $k \leq n$ we  have $\Phi(V)=(\varphi,\pi/2,\pi/2)$, for some $\varphi\in[0,\pi/2]$.
	
	\item \emph{Case $k=3$}. By Proposition~\ref{prop:hairyball}, $\Phi(V)=(\varphi,\varphi,\pi/2)$ for some $\varphi\in[0,\pi/2]$. If $n\ge 3$, Propositions \ref{prop:sk=3} and \ref{prop:ineqk=3} guarantee that, for each triple $(\varphi,\varphi,\pi/2)$ in $\mathfrak{R}^-_{3}$ (resp.\ in $\mathfrak{R}^+_{3}\setminus\g{R}^-_3$) we have exactly two (resp.\ one) inequivalent subspaces with  $\Phi(V)=(\varphi,\varphi,\pi/2)$. By Remark \ref{rem:k3smaller_n}, if $n=2$,  we have $\Phi(V)\in\{(0,0,\pi/2),(\pi/3,\pi/3,\pi/2)\}$, whereas if $n=1$, $\Phi(V)=(0,0,\pi/2)$.\qedhere
\end{enumerate}
\end{proof}

\begin{proof}[Proof of Theorem~B]
This follows from combining Theorem~A with the theory of cohomogeneity one actions on symmetric spaces of non-compact type and rank one (cf.~\S\ref{sec:cohomo}).
We just have to note that the action producing the solvable foliation (resp.\ the action with a totally geodesic singular orbit $\H H^\ell$, $\ell\in\{1,\dots,n\}$) can be recovered by the method that yields the actions with a non-totally geodesic singular orbit by taking $V$ as a $1$-dimensional subspace of $\g{g}_\alpha\cong\H^n$ (resp.\ by taking $V$ as a quaternionic subspace $\H^{n-\ell+1}$ in $\g{g}_\alpha\cong\H^n$).
\end{proof}


\end{document}